\documentclass[12pt,reqno]{amsart}
\usepackage{mathrsfs}
\usepackage{amsgen}
\usepackage{amscd}
\usepackage{amsmath}
\usepackage{latexsym}
\usepackage{amsfonts}
\usepackage{amssymb}
\usepackage{amsthm}
\usepackage{graphicx}
\usepackage{color}
\usepackage{amsthm,amssymb}
\usepackage{graphicx}
\usepackage{enumerate}
\usepackage{amssymb,amsmath,graphicx,amsfonts,euscript}
\usepackage{color}
\usepackage{mathrsfs}

\usepackage[margin=3cm, a4paper]{geometry}

\def\H{{\cal H}}

\def\R{\mathbb{R}}
\def\Z{\mathbb{Z}}
\def\T{\mathbb{T}}
\def\C{\mathbb{C}}

\def\H2{H^2(\R^N)}
\def\L2{L^2(\R^N)}
\def\to{\rightarrow}

\def\H{{\cal H}}

\def\H1{H^1(\R)}

 \newcommand{\Del}[1]{}

\numberwithin{equation}{section}

\newtheorem{thm}{Theorem}[section]
\newtheorem{cor}[thm]{Corollary}
\newtheorem{lem}[thm]{Lemma}
\newtheorem{assum}[thm]{Assumption}
\newtheorem{prop}[thm]{Proposition}
\newtheorem{definition}[thm]{Definition}

\theoremstyle{remark}
\newtheorem{remark}[thm]{Remark}
\newtheorem*{exam*}{Examples}

\begin{document}

\setcounter{page}{1}

\title[Large global solutions for NLS]{Large global solutions for nonlinear Schr\"odinger equations III, energy-supercritical cases}

\author{Marius Beceanu}
\address{Department of Mathematics and Statistics\\
University at Albany SUNY\\
Earth Science 110\\
Albany, NY, 12222, USA\\}
\email{mbeceanu@albany.edu}
\thanks{}

\author{Qingquan Deng}
\address{Department of Mathematics\\
Hubei Key Laboratory of Mathematical Science\\
Central China Normal University\\
Wuhan 430079, P. R. China\\}
\email{dengq@mail.ccnu.edu.cn}
\thanks{}

\author{Avy Soffer}
\address{Department of Mathematics\\
Hubei Key Laboratory of Mathematical Science\\
Central China Normal University\\
Wuhan 430079, P. R. China, and Department of Mathematics\\
Rutgers University\\
110 Frelinghuysen Rd.\\
Piscataway, NJ, 08854, USA\\}
\email{soffer@math.rutgers.edu}
\thanks{}

\author{Yifei Wu}
\address{Center for Applied Mathematics\\
Tianjin University\\
Tianjin 300072, P. R. China}
\email{yerfmath@gmail.com}
\thanks{}

\subjclass[2010]{Primary  35Q55}


\keywords{Nonlinear Schr\"{o}dinger equation,
global well-posedness, scattering}

\maketitle

\begin{abstract}\noindent
In this work,  we mainly focus on the energy-supercritical
nonlinear Schr\"odinger equation,
$$
    i\partial_{t}u+\Delta u= \mu|u|^p u, \quad (t,x)\in \R^{d+1},
$$
with $\mu=\pm1$ and $p>\frac4{d-2}$.

We prove that for radial initial data with high frequency, if it is outgoing (or incoming) and in rough space
$H^{s_1}(\R^d)$ $(s_1<s_c)$ or its Fourier transform belongs to $W^{s_2,1}(\R^d)$ $(s_2<s_c)$,
the corresponding solution is global and scatters forward (or backward) in time.
We also construct a class of large global and scattering solutions starting with many bubbles, which are mingled with in the physical space and separate in the frequency space. The analogous results are also valid for the energy-subcritical cases.


\end{abstract}

\tableofcontents

\section{Introduction}
We study the Cauchy problem for the following nonlinear
Schr\"{o}dinger equation (NLS) on $\R\times\R^d$:
 \begin{equation}\label{eqs:NLS-cubic}
   \left\{ \aligned
    &i\partial_{t}u+\Delta u=\mu |u|^p u,
    \\
    &u(0,x)  =u_0(x),
   \endaligned
  \right.
 \end{equation}
with $\mu=\pm1$ and $p>0$.
Here $u(t,x):\R\times\R^d\rightarrow \C$ is a complex-valued function. $\mu=1, -1$ denotes the nonlinearity is defocusing and focusing, respectively.
The class of solutions to equation (\ref{eqs:NLS-cubic}) is invariant under the scaling
\begin{equation}\label{eqs:scaling-p}
u(t,x)\to u_\lambda(t,x) = \lambda^{\frac2p} u(\lambda^2 t, \lambda x) \ \ {\rm for}\ \ \lambda>0,
\end{equation}
which maps the initial data as
\begin{eqnarray}
u(0)\to u_{\lambda}(0):=\lambda^{\frac2p} u_{0}(\lambda x) \ \ {\rm for}\ \ \lambda>0.\nonumber
\end{eqnarray}
Denote
$$
s_c=\frac d2-\frac2p.
$$
Then the scaling  leaves  $\dot{H}^{s_{c}}$ norm invariant, that is,
\begin{eqnarray*}
\|u\|_{\dot H^{s_{c}}}=\|u_{\lambda}\|_{\dot H^{s_{c}}},
\end{eqnarray*}
which is called \emph{critical regularity} $s_{c}$. It is also considered as the lowest regularity that problem  (\ref{eqs:NLS-cubic}) is well-posed for general $H^{s}(\R^d)$-data, since one can always find  some special initial datum belonging to $H^s(\R^d), s<s_c$ such that the problem   (\ref{eqs:NLS-cubic}) is ill-posed. Note that $\dot H^{s_c}(\R^d)\hookrightarrow L^{p_c}(\R^d)$ with $p_c=\frac{dp}{2}$, and
\begin{eqnarray*}
\|u\|_{L^{p_c}}=\|u_{\lambda}\|_{L^{p_c}},
\end{eqnarray*}
then one naturally takes $L^{p_c}(\R^d)$ as the critical Lebesgue space.

The $H^1$-solution of equation \eqref{eqs:NLS-cubic} also enjoys  the mass, momentum and energy
conservation laws, which read
\begin{equation}\label{eqs:energy-mass}
   \aligned
M(u(t))&:=\int |u(x,t)|^2\,dx=M(u_0),\\
P(u(t))&:=\textrm{Im}\int \overline{u(x,t)}\nabla u(x,t)\,dx=P(u_0),\\
E(u(t)) &:= \int |\nabla u(x,t)|^2\,dx + \frac{2\mu}{p+2}\int
|u(x,t)|^{p+2} \,dx = E(u_0).
   \endaligned
\end{equation}

The well-posedness and scattering theory for Cauchy problem (\ref{eqs:NLS-cubic}) with initial data in $H^{s}(\R^d)$ are extensively studied.
The local well-posedness theory follows from a standard fixed point argument, implying that  for all $u_{0}\in H^{s}(\R^d)$ with $s\ge s_c$, there exists $T_{0}>0$ such that its corresponding solution $u\in C([0,T_{0}),H^{s}(\R^d))$. In fact, the above $T_{0}$ depends on $\|u_{0}\|_{H^{s}(\R^d)}$ when $s>s_c$ and also the profile of $u_{0}$  when $s=s_c$. Some of the results can be found in Cazenave and Weissler \cite{CW1}. Such argument can be applied directly to prove the global well-posedness for solutions to equation (\ref{eqs:NLS-cubic})  with small initial data in $H^{s}(\R^d)$ with $s\geq s_c$. It is of great interest  to consider the large initial data problem of NLS for  supercritical case $s_{c}>1$, since all of known conservations are below the critical scaling regularity, few of the results on the long time behavior of the large data solutions were established, even the initial datum are smooth enough.

Recently, conditional global and scattering results with assumption that
$$u\in L^\infty_t(I,\dot H^{s_c}_x(\R^d))$$
 were considered by many authors, which was started from \cite{KeMe-cubic-NLS-2010, KeMe-wave-2011-1} and then developed by
\cite{Bu, DoMiMuZh-17,DKM, DuRo, KeMe-wave-2011-2, KiMaMuVi-NoDEA-2017, KiMaMuVi-2018, KV, KV-10, KV-10-3, MJJ, MWZ, Mu, Mu-3, XieFa-13} and references therein.
An important result indicated from these works is that
if the initial data $u_{0}\in \dot H^{s_c}(\R^d)$ and the solution has priori estimate
\begin{eqnarray}
\sup_{0<t<T_{out}(u_{0})}\|u\|_{\dot H^{s_c}_x(\R^d)}<+\infty, \label{uniformbound}
\end{eqnarray}
then $T_{out}(u_{0})=+\infty$ and the solution scatters in $\dot H^{s_c}(\R^d)$; here $[0,T_{out}(u_{0}))$ is the forward maximal interval for existence of the solution. Consequently, these results give the blowup criterion that the lifetime only depends on the critical norm $\|u\|_{L^\infty_t\dot H^{s_c}_x(\R^d)}$.  


In this paper, we  consider the large global solution for the energy-supercritical nonlinear Schr\"odinger equation.  The results obtained here are unconditional ones and are also valid for the energy-subcritical cases. It is the third part of our series of works on large global and scattering solution of nonlinear Schr\"odinger equation. In fact,
in the first part of our series of works \cite{BDSW}, we considered the global solution for the mass-subcritical nonlinear Sch\"odinger equation in the critical space $\dot H^{s_c}(\R^d)$ and proved that  for radial initial data with compact support in space, the corresponding solution is global in time in $\dot H^{s_c}(\R^d)$. In the second part of our series of works \cite{BDSW-II}, we considered the global solution for the the defocusing mass-supercritical, energy-subcritical nonlinear Sch\"odinger equation in the supercritical space $\dot H^{s}(\R^d), s<s_c$ and  proved that
under some restrictions on $d$ and $p$, there exists $s_0<s_c$, such that any function in  $H^{s_0}(\R^d)$ with the support away from the origin, it has an incoming/outgoing decomposition. Moreover, the outgoing part of the initial data  leads to the global well-posedness and scattering forward in time; while the incoming part of initial data leads to the global well-posedness and scattering backward in time.

The literature for the  energy-supercritical is very limited compared to the energy-critical and subcritical cases. As mentioned previously, the main reason is the lack of the conservation laws beyond $H^1(\R^d)$ Sobolev space.
Besides the conditional global results described above, the unconditional results on the energy-supercritical equations are mostly from the wave equations, see \cite{BS, BS-2, GeGr-1, GeGr-2, KS, IbMoMa, Li-Skyrme, MiPeYu, Tao, Roy, Roy-1, St, WaYu, Yan} and cited references. The first results from Tao \cite{Tao}, who proved global well-posedness and scattering for radial initial data for the logarithmically
supercritical defocusing wave equation. The proof employs the Morawetz estimate and time slicing argument which is critical in $H^1(\R^d)$, following \cite{GSV}.  Then Roy \cite{Roy, Roy-1} further proved the scattering of solutions to the log-log-supercritical  defocusing wave equations. Recently, Bulut and Dodson \cite{BuDo} extended the work of Tao in the radially symmetric setting to a partially symmetric setting. Struwe \cite{St} proved the global well-posedness for the exponential nonlinear wave equation in two dimension.
Furthermore, it was first proved by Li \cite{Li-Skyrme} that the $(3+1)$-Skyrme problem, which is energy-supercritical,  is globally well-posed with hedgehog solutions for arbitrary large initial data. The work was followed by Geba and Grillakis \cite{GeGr-1,GeGr-2} to study the classical equivariant Skyrme model and the $2+ 1$-dimensional equivariant Faddeev model, both of which are the energy-supercritical models.
Very recently, large outgoing solution for the nonlinear  wave equation was constructed in the papers of the first and  the third authors \cite{BS, BS-2}, by using the explicit formula of the outgoing and incoming components of the radial linear wave flow in three dimension.  But such formulism does not exist for the Schr\"odinger equation.
See also \cite{Do-14,Do-14-2,DoSch-14-1, Collot-2018} for the blowing up results for the energy-supercritical wave equations.
When it comes to the Schr\"odinger setting, the numerical investigation is from \cite{CSS}, who considered the equation,
$$i\partial_tu+\Delta u=|u|^4u, \ \ x\in \R^5,$$
and found a class of global solutions which are large and uniformly bounded in $H^2(\R^5)$. Yet, Tao \cite{Tao-4} showed that there exist a class of defocusing nonlinear Schr\"odinger systems which the solutions
can blow up in finite time by suitably constructing the initial datum (see also \cite{Tao-3} for the analogous results for the wave equations). In the focusing case, Merle,  Rapha\"el, and Rodnianski \cite{MRR} shows the existence of type II blowing-up solutions  for the energy-supercritical nonlinear Schr\"odinger equation.
 See also Wang \cite{Wang}, who constructed a class of quasi-periodic solutions to the energy-supercritical nonlinear Schr\"odinger equation  which are small in $\dot H^{s_c}(\T^d)$.

In this paper,
we construct classes of initial data, which could be arbitrarily large in critical Sobolev space $\dot H^{s_c}(\R^d)$ such that its corresponding solutions of \eqref{eqs:NLS-cubic} exist globally in time and scatter. In particular, we construct the global large solutions verifying the result in \cite{CSS} theoretically.
To state the first result, we introduce  two indices as follows which are denoted by
$$
s_1=\max\{\frac{s_c}{d}-\varepsilon_0, s_c-\frac{4d-1}{4d-2}+\frac2d+\varepsilon_0\},
$$
and
$$
s_2=\max\{-\varepsilon_0, s_c-\frac{d-2}{2(d-1)}-\frac{d-1}{2d-1}+\varepsilon_0\},
$$
where  $\varepsilon_0>0$ is a small fixed constant.
Denote by $\hat W^{s,1}(\R^d)$ the space of functions such that
\begin{align*}
\|h\|_{\hat W^{s,1}(\R^d)}:=\big\|\langle\xi\rangle^s\hat h\big\|_{L^1(\R^d)}
 \end{align*}
 it finite.
Then we have the following result.
\begin{thm}\label{thm:main2}
Let $p\ge\frac4d$, $d=3,4,5$ and $\mu=\pm1$.
Assume that there exists a small constant $\delta_{0}$ such that
the radial function $f$ satisfies
\begin{align}
\big\|\chi_{\le 1}f\big\|_{\dot H^{s_c}(\R^d)}+\|\chi_{\ge 1}f\|_{H^{s_1}(\R^d)} \le \delta_0;
\label{f-cond}
\end{align}
and the radial function $g$ satisfies
\begin{align}
\mbox{supp }g\in \{x:|x|\le 1\},\quad \mbox{and }\quad \big\|\langle\xi\rangle^{s_2}\hat g\big\|_{L^1(\R^d)}\le \delta_0.
\label{g-cond}
\end{align}
Then if the initial date of equation \eqref{eqs:NLS-cubic} is of form
$$
u_0=f_{+}+g \qquad  (\mbox{or} \quad u_0=f_{-}+g),
$$
the corresponding solution $u$
exists globally forward (or backward) in time and
$$u\in C(\R^+;H^{s_1}(\R^d)+\hat W^{s_2,1}(\R^d)) \  \  (or\  u\in C(\R^-;H^{s_1}(\R^d)+\hat W^{s_2,1}(\R^d))).$$
Here $f_{+}$ and $f_{-}$ are the modified outgoing and  incoming components of $f$ respectively, which are  defined in Definition \ref{def:+-}, with
$$
f=f_++f_-.
$$
Moreover, there exists $u_{0+}\in H^{s_1}(\R^d)+\hat W^{s_2,1}(\R^d)$ (or $u_{0-}\in H^{s_1}(\R^d)+\hat W^{s_2,1}(\R^d)$), such that when $t\to +\infty$ (or $t\to -\infty$),
\begin{align}
\lim\limits_{t\to +\infty}\|u(t)-e^{it\Delta}u_{0+}\|_{\dot H^{s_c}(\R^d)}= 0 \qquad
(\mbox{or } \lim\limits_{t\to -\infty}\|u(t)-e^{it\Delta}u_{0-}\|_{\dot H^{s_c}(\R^d)}= 0). \label{scattering2}
\end{align}
\end{thm}

\begin{remark}\label{rem:1}
%
%
{It is worth noting that $s_1<s_c$, which means that $f_+$ (or $f_-$) in initial data can be localized in the supercritical Sobolev space $\dot H^s(\R^d)$ for some $s<s_c$. Indeed, if $f\in \dot H^s(\R^d)$, we have the fact that
      at least one of $f_{+}$ and
$f_{-}$ belongs to $\dot H^s(\R^d)$ since  $$f=f_{+}+f_{-}.$$
In particular, if we only consider the high frequency of $f$, then $f$ can be arbitrary large in critical Sobolev space $\dot H^{s_c}(\R^d)$. Moreover, since $s_1<\frac{s_c}{d}$ when $s_c$ is sufficient small and $\|f\|_{L^{p_c}(\R^d)}$ can not be controlled by $\|f\|_{H^{s_1}(\R^d)} $ even $f$ is radial and supported away from origin, we have that  $L^{p_c}(\R^d)$ could also be arbitrary large. }

Let $
a_0=\max\{\frac{s_c(d-1)}{d}, \frac{4d-1}{4d-2}-\frac2d-2\varepsilon_0\}$ be a positive constant.
We have the following corollary.

\begin{cor}
Let $p\ge\frac4d$, $d=3,4,5$ and $\mu=\pm1$.
Then there exist $N_0>0$ and a small constant $\delta_{0}$, such that for  given $N\ge N_0$ and any radial function $f$ satisfying
\begin{align*}
\big\|\chi_{\le 1}f\big\|_{\dot H^{s_c}(\R^d)}+\|P_{\le N}(\chi_{\geq 1}f)\|_{\dot H^{s_c}(\R^d)} \le \delta_0,\quad \|P_{\ge N}(\chi_{\geq 1}f)\|_{\dot H^{s_c}(\R^d)} \le N^{a_0},
\end{align*}
the solution $u$ to the equation \eqref{eqs:NLS-cubic} with the initial data
$$
u_0=f_{+} \qquad  (\mbox{or} \quad u_0=f_{-})
$$
exists globally forward (or backward) in time and
 $$u\in C(\R^+;H^{s_c}(\R^d)) \ \ \  (or\  u\in C(\R^-;H^{s_c}(\R^d))).$$
Moreover, there exists $u_{0+}\in H^{s_c}(\R^d)$ (or $u_{0-}\in H^{s_c}(\R^d)$), such that when $t\to +\infty$ (or $t\to -\infty$),
\begin{align*}
\lim\limits_{t\to +\infty}\|u(t)-e^{it\Delta}u_{0+}\|_{\dot H^{s_c}(\R^d)}= 0 \qquad
(\mbox{or } \lim\limits_{t\to -\infty}\|u(t)-e^{it\Delta}u_{0-}\|_{\dot H^{s_c}(\R^d)}= 0).
\end{align*}
\end{cor}

On the other hand, we note that  $s_2<s_c$ in Theorem \ref{thm:main2}. The function $g$ in initial data is localized in  $\hat W^{s_{2},1}(\R^d)$, which needs  much less derivatives than $\dot H^{s_c}(\R^d)$. Moreover, one notices that the initial data can be large in
$L^\infty(\R^d)$ if  $p$ is   close to $\frac4d$ and  $s_2<0$. Indeed, we have the following corollary.

\begin{cor}\label{cor1-4}
Let $d=3,4,5$, $0\le s_c<\frac{d-2}{2(d-1)}+\frac{d-1}{2d-1}$ and $\mu=\pm1$.
Then there exist $N_{0}$ and $b_0>0$, such that for given $N>N_{0}$  any  radial function $g$ satisfying
\begin{align}
\mbox{supp }g\in \{x:|x|\le 1\}, \quad g=P_{\ge N}g,\quad \mbox{and }\quad \big\|\hat g\big\|_{L^1(\R^d)}\le N^{b_0},
\end{align}
 the solution $u$ to the equation \eqref{eqs:NLS-cubic} with the initial data
$
u_0=g
$
exists globally in time and $u\in C(\R;\hat W^{s_2,1}(\R^d))$.
Moreover, there exists $u_{0\pm}\in \hat W^{s_2,1}(\R^d)$, such that when $t\to \pm\infty$,
\begin{align*}
\lim\limits_{t\to \pm\infty}\|u(t)-e^{it\Delta}u_{0+}\|_{\dot H^{s_c}(\R^d)}= 0.
\end{align*}
\end{cor}
\end{remark}
Notice that in Corollary \ref{cor1-4}, we require $s_{c}<\frac{d-2}{2(d-1)}+\frac{d-1}{2d-1}$. In fact, the result is also valid for $s_{c}\ge\frac{d-2}{2(d-1)}+\frac{d-1}{2d-1}$ if one impose extra regularity on $g$. This corollary applies also in the focusing case, and therefore, the argument based on pseudo-conformal identity (that uses the localization of the initial data, as done by Bourgain \cite{Bou-book}) may not work.

The key ingredients we rely on for the proof of Theorem  \ref{thm:main2} are as follows.
The first one is the estimates obtained in Section \ref{sec:LE-Out-in} below, which regards as the decomposition of the incoming and outgoing waves and their supercritical space-time estimates.  In particular, the estimates obtained imply that the incoming/outgoing solution has the ``smoothing effect'' as follows:  any $\varepsilon>0$,
\begin{align*}
\left\|e^{it\Delta}  \big(\chi_{\ge 1}f\big)_{out}\right\|_{L^2_tL^\infty_x(\R^{+}\times\R^d)}
\lesssim
\|\chi_{\ge 1}f\|_{H^{\varepsilon}(\R^d)}.
\end{align*}
The second one is that if the function $g$ is compactly supported, then we have the estimates with the ``smoothing effect" in the following sense: any $q\ge 2$,
\begin{align*}
\big\|e^{ it\Delta}g\big\|_{L^q_tL^\infty_x(\R\times\R^d)}\lesssim \big\|\langle \xi\rangle^{-\frac{d-2}{q(d-1)}+}\hat g\big\|_{L^1_\xi(\R^d)}.
\end{align*}
Then we consider the equation for $w=u-v_L$ where $u$ is the solution to equation (\eqref{eqs:NLS-cubic}) and $v_L$ is the linear solution with the initial data $f_{+}+g$. It is easy to see that
$w$ obeys the equation of
\begin{equation*}
    i\partial_{t}w+\Delta w=\mu |u|^pu.
 \end{equation*}
By using bootstrap argument and the space-time estimates for $v_L$, we could prove Theorem  \ref{thm:main2}. We also note that the choice of working spaces to this problem would be another obstacle, since one may find that the space-time estimates are not standard.

Now we  introduce our second main result, which is about the global solution for initial data which consists of many bubbles, which are mingled within the physical space but separate  in the frequency space.

Before stating our result, we introduce  the hypothesis on the initial data.
 \begin{assum}\label{ass-f-2}
Given a constant $\epsilon\in (0,1]$. We assume that $h$  is of form
$$
h=\sum\limits_{k=0}^{+\infty} h_k,
$$
with
$$
\mbox{supp }\widehat{h_k}=\{\xi:2^k\le |\xi|\le (1+\epsilon)2^k\}.
$$
where  $h_k\in \dot H^{s_c}(\R^d)$ and there exists an absolute constant $\alpha_0>0$, such that
\begin{align}
\|h\|_{\dot H^{s_c}(\R^d)}\le \epsilon^{-\alpha_0}.\label{ass-f}
\end{align}
\end{assum}

Now our second main theorem is stated as follows.

\begin{thm}\label{thm:main3}
Let $d\ge1, p> \frac4d$ and $\mu=\pm1$. Then
there exists  some constant  $\epsilon_0\in (0,1]$ such that  $h$ satisfies Assumption \ref{ass-f-2} with respect to $\epsilon$
and $\psi_0\in \dot H^{s_c}(\R^d)$  satisfies
\begin{align}
\|\psi_0\|_{\dot H^{s_c}(\R^d)}\le \epsilon\label{10.02}
\end{align}
for  $\epsilon\in (0,\epsilon_0]$,
the solution $u$ to the equation \eqref{eqs:NLS-cubic} with the initial data
$$
u_0=\psi_0+h
$$
exists globally in time and
$
u\in C_t\dot H^{s_c}_x(\R\times\R^d).
$
Furthermore, there exists $u_{0\pm}\in \dot H^{s_c}(\R^d)$ such that 
\begin{align*}
\lim\limits_{t\to \pm\infty}\|u(t)-e^{it\Delta}u_{0\pm}\|_{\dot H^{s_c}(\R^d)}= 0.
\end{align*}
\end{thm}

\begin{remark}\label{rem:2}
Under Assumption \ref{ass-f-2},
the function $h$ in initial data can be regarded as a combination of the bubbles $h_k$ which are separated in frequency space. Hence, $h$ can be arbitrary large in $\dot H^{s_c}(\R^d)$ by choosing  $\epsilon$ small enough. We take one bubble case as an example, 
for any arbitrary large $L$, let
$$
\hat h(\xi)=\epsilon^{-\frac12+\frac{\alpha_0}{2}}\chi_{\le1}\Big(\frac{|\xi|-1}{\epsilon}\Big),
$$
then  $\|h\|_{\dot H^{s_c}(\R^d)}= L$,  for $\epsilon\sim L^{-\frac1{\alpha_0}}$. 
\end{remark}

The key observation in the proof of Theorem \ref{thm:main3} is based on the following estimate
\begin{align*}
\big\||\nabla|^{\gamma} e^{it\Delta}h_k\big\|_{L^{q_\gamma}_{tx}(\R\times \R^d)}
\lesssim \epsilon^{\frac{s_c-\gamma}{d}}\|h_k\|_{\dot H^{s_c}}
\end{align*}
 with $\frac{d+2}{q_\gamma}-\gamma=\frac d2-s_c$. Since the norms are scaling invariant, we believe that the above estimate is nontrivial. The key observation is due to the weak topology of the space-time norm compared to the Sobolev norm, and the narrow belt restriction on the frequency.

\textbf{Organization of the paper.}  In Section 2, we give some preliminaries. This includes
some basic lemmas, some estimates on the linear Schr\"odinger operator. Moreover, we recall the definition of the incoming/outgoing waves and their basic properties which were obtained in our previous paper \cite{BDSW-II}.
In Sections 3 and 4, we establish the estimates on linear flow. In Section 5, we give some spacetime estimates on many bubbles case.
In Section 6, we give the proof of the main
theorems.


\section{Preliminary}

\subsection{Notations}

We write $X \lesssim Y$ or $Y \gtrsim X$ to indicate $X \leq CY$ for some constant $C>0$. If $C$ depends upon some additional
parameters, we will indicate this with subscripts; for example, $X\lesssim_a Y$ denotes the
assertion that $X\le C(a)Y$ for some $C(a)$ depending on $a$. We use $O(Y)$ to denote any quantity $X$
such that $|X| \lesssim Y$.  We use the notation $X \sim Y$ whenever $X \lesssim Y \lesssim X$.

For small constant $\epsilon>0$, the notations $a+$ and $a-$ denote $a+\epsilon$ and $a-\epsilon$, repectively.
The fractional derivative is given by $|\nabla|^\alpha=(-\partial^2_x)^{\alpha/2}$.
Denote by $\mathcal S(\R^d)$ the Schwartz function space on $\R^d$ and $\mathcal S'(\R^d)$ its topological dual space.  Let $h\in \mathcal S'(\R^{d+1})$, we use
$\|h\|_{L^q_tL^p_x}$ to denote the mixed norm
$\Big(\displaystyle\int\|h(\cdot,t)\|_{L^p}^q\
dt\Big)^{\frac{1}{q}}$, and $\|h\|_{L^q_{xt}}:=\|h\|_{L^q_xL^q_t}$. Sometimes, we use the notation $q'=\frac{q}{q-1}$.

Throughout this paper, we use $\chi_{\le a}$ for $a\in \R^+$ to be the smooth function
\begin{align*}
\chi_{\le a}(x)=\left\{ \aligned
1, \ & |x|\le a,\\
0,    \ &|x|\ge \frac{11}{10} a.
\endaligned
  \right.
\end{align*}
Moreover, we denote $\chi_{\ge a}=1-\chi_{\le a}$ and $\chi_{a\le \cdot\le b}=\chi_{\le b}-\chi_{\le a}$. We denote $\chi_{a}=\chi_{\le 2a}-\chi_{\le a}$ and $\chi_{\thicksim a}=\chi_{\frac{1}{2}a\le\cdot\le4a}$ for short.
%
For any interval $\Omega\subset \R$, we denote $\mathbb I_\Omega$ as its characteristic function
\begin{align*}
\mathbb I_\Omega(x)=\left\{ \aligned
1, \ & x\in \Omega,\\
0,    \ &x\not\in \Omega.
\endaligned
  \right.
\end{align*}

Also, we need some Fourier operators.
For each number $N > 0$, we define the Fourier multipliers $P_{\le N}, P_{> N}, P_N$ as
\begin{align*}
\widehat{P_{\leq N} f}(\xi) &:= \chi_{\leq N}(\xi) \hat f(\xi),\\
\widehat{P_{> N} f}(\xi) &:= \chi_{> N}(\xi) \hat f(\xi),\\
\widehat{P_N f}(\xi) &:= \chi_{N}(\xi) \hat
f(\xi),
\end{align*}
and similarly $P_{<N}$ and $P_{\geq N}$.  We also define
$$
\tilde P_N := P_{N/2} + P_N +P_{2N}.
$$
 We will usually use these multipliers when $N$ are \emph{dyadic numbers} (that is, of the form $2^k$
for some integer $k$).

\subsection{Basic lemmas}
First of all, we introduce the following Sobolev embedding theorem for radial function, see \cite{TaViZh} for example.
\begin{lem}\label{lem:radial-Sob}
Let $\alpha,q,p,s$ be the parameters which satisfy
$$
\alpha>-\frac dq;\quad \frac1q\le \frac1p\le \frac1q+s;\quad 1\le p,q\le \infty; \quad 0<s<d
$$
with
$$
\alpha+s=d(\frac1p-\frac1q).
$$
Moreover, at most one of the equalities hold:
$$
p=1,\quad p=\infty,\quad q=1,\quad q=\infty,\quad \frac1p=\frac1q+s.
$$
Then for any radial function $u$ such that $ |\nabla|^s u\in L^p(\R^d)$,
\begin{align*}
\big\||x|^\alpha u\big\|_{L^q(\R^d)}\lesssim \big\||\nabla|^su\big\|_{L^p(\R^d)}.
\end{align*}
\end{lem}

The second lemma is the following fractional Leibniz rule, see \cite{BoLi-KatoPonce, KePoVe-CPAM-1993, Li-KatoPonce} and references therein.
\begin{lem}\label{lem:Frac_Leibniz}
Let $0<s<1$, $1<p\le \infty$, and $1<p_1,p_2,p_3, p_4 \le \infty$ with $\frac1p=\frac1{p_1}+\frac1{p_2}$, $\frac1p=\frac1{p_3}+\frac1{p_4}$, and let $f,g\in \mathcal S(\R^d)$,  then
\begin{align*}
\big\||\nabla|^s(fg)\big\|_{L^p}\lesssim \big\||\nabla|^sf\big\|_{L^{p_1}}\|g\|_{L^{p_2}}+ \big\||\nabla|^sg\big\|_{L^{p_3}}\|f\|_{L^{p_4}}.
\end{align*}
\end{lem}

Consequently, we have the  following elementary inequality, one can see  \cite{BDSW} for the proof.
\begin{lem}\label{lem:frac_Hs}
For any $a>0, 1<p\le \infty, 0\le \gamma<\frac dp$, and $ |\nabla|^\gamma g\in L^p(\R^d)$,
\begin{align}
\big\||\nabla|^\gamma\big(\chi_{\le a}g\big)\big\|_{L^p(\R^d)}\lesssim \big\||\nabla|^\gamma g\big\|_{L^p(\R^d)}. \label{15.37}
\end{align}
Here the implicit constant is independent on $a$.
The same estimate holds for $\chi_{\ge a}g$.
\end{lem}

We need the following mismatch result, which is  helpful in commuting the spatial and the frequency cutoffs.
\begin{lem}[Mismatch estimates,  \cite{LiZh-APDE}]\label{lem:mismatch}
Let $\phi_1$ and $\phi_2$ be smooth functions obeying
$$
|\phi_j| \leq 1 \quad \mbox{ and }\quad \mbox{dist}(\emph{supp}
\phi_1,\, \emph{supp} \phi_2 ) \geq A,
$$
for some large constant $A$.  Then for $\sigma>0$, $M\le 1$ and $1\leq r\leq
q\leq \infty$,
\begin{align}
\bigl\| \phi_1 |\nabla|^\sigma P_{\leq M} (\phi_2 f)
\bigr\|_{L^q_x(\R^d)}
    &+ \bigl\| \phi_1 \nabla |\nabla|^{\sigma-1} P_{\leq M} (\phi_2 f) \bigr\|_{L^q_x(\R^d)}
    \lesssim A^{-\sigma-\frac dr + \frac dq} \|\phi_2 f\|_{L^r_x(\R^d)};\label{eqs:lem-mismath-1}\\
    \bigl\| \phi_1   P_{\leq M} (\phi_2 f)
\bigr\|_{L^q_x(\R^d)}&\lesssim_m M^{-m}A^{-m}\|f\|_{L^q_x(\R^d)}, \mbox{
for any } m\ge0.\label{eqs:lem-mismath-2}
\end{align}
\end{lem}

Furthermore, we also need the following Leibnitz-type formula, which is proved in \cite{BDSW}.
\begin{lem}\label{lem:muli-Lei-formula}
Let  $f\in \mathcal S(\R^d)^d$ be a vector-valued function and  $g\in \mathcal S(\R^d)$ be a scalar function. Then for any integer $N$,
\begin{equation*}
\nabla_{\xi}\cdot \big(f\>\nabla_{\xi} \big)^{N-1}\cdot
(fg)=\sum\limits_{\begin{subarray}{c}
l_1,\cdots,l_N\in\R^d,l'\in\R^d;\\
|l_j|\le
j;|l_1|+\cdots+|l_N|+|l'|=N
\end{subarray}}
C_{l_1,\cdots,l_N,l'}\partial_\xi^{l_1}f\cdots
\partial_\xi^{l_N}f\>\partial_\xi^{l'}g,
\end{equation*}
where we have used the notations
$$
\nabla_{\xi}=\{\partial_{\xi_1},\cdots,\partial_{\xi_d}\};\quad
\partial_\xi^{l}=\partial_{\xi_1}^{l^1}\cdots\partial_{\xi_d}^{l^d}, \mbox{ for any } l=\{l^1,\cdots,l^d\}\in \R^d.
$$
\end{lem}

\subsection{Linear Schr\"odinger operator}

Let the operator $S(t)=e^{it\Delta}$ be the linear Schr\"odinger flow, that is,
$$
(i\partial_t+\Delta)S(t)\equiv 0.
$$
The following are some fundamental properties of the operator $e^{it\Delta}$. The first is the explicit formula, see for example \cite{Cazenave-book}.
\begin{lem}\label{lem:formula-St}
For all $\phi\in \mathcal S(\R^d)$, $t\neq 0$,
$$
S(t)\phi(x)=\frac{1}{(4\pi it)^{\frac d2}}\int_{\R^d} e^{\frac{i|x-y|^2}{4t}}\phi(y)\,dy.
$$
Moreover, for any $r\ge2$,
$$
\|S(t)\phi\|_{L^r_x(\R^d)}\lesssim |t|^{-d(\frac12-\frac1r)}\|\phi\|_{L^{r'}(\R^d)}.
$$
\end{lem}

The following is the  standard Strichartz estimate, see for example \cite{KeTa-Strichartz}.
\begin{lem}\label{lem:strichartz}
Let $I$ be a compact time interval and let $u: I\times\R^d \to
\mathbb \R$ be a solution to the inhomogeneous Schr\"odinger equation
$$
iu_{t}- \Delta u + F = 0.
$$
Then for any $t_0\in I$,  any pairs $(q_j,r_j), j=1,2$ satisfying
$$
 q_j\ge 2, \,\, r_j\ge 2,\,\,  \mbox{and }\,\, \frac2{q_j}+\frac d{r_j}=\frac d2,
$$
the following estimates hold,
\begin{align*}
\bigl\|u\bigr\|_{C(I;L^2(\R^d))}+\big\|u\big\|_{L^{q_1}_tL^{r_1}_x(I\times\R^d)}
\lesssim \bigl\|u(t_0)\bigr\|_{L^2_x(\R^d)}+
\bigl\|F\bigr\|_{L^{q_2'}_tL^{r_2'}_x(I\times\R^d)}.
\end{align*}
\end{lem}

Further, we need the following  inhomogeneous Strichartz estimate which is not sharp but sufficient for this paper, see for examples \cite{Cazenave-book} and \cite{KeTa-Strichartz}.
\begin{lem}\label{lem:strichartz-inhomog}
Let $I$ be a compact time interval and $t_0\in I$. Assume that $2<r<\frac{2d}{d-2}$ $(2<r\le \infty \; if \; d=1)$, and $1<q,\tilde q<\infty$ satisfy
$$
\frac1q+\frac1{\tilde q}=d(\frac12-\frac1r).
$$
Then the following estimates hold,
\begin{align*}
\Big\|\int_0^t e^{i(t-s)\Delta}F(s)\,ds\Big\|_{L^q_{t}L^r_x(I\times\R^d)}
\lesssim
\bigl\|F\bigr\|_{L^{\tilde{q}'}_tL^{r'}_x(I\times\R^d)}.
\end{align*}
\end{lem}

We also need the special Strichartz estimate for radial data, which was first proved in \cite{Shao} and then developed in \cite{CL,GW}.
\begin{lem}[Radial Strichartz estimates]\label{lem:radial-Str}
Let $g\in L^2(\R^d)$ be a radial function, and let the triple $(q,r,\gamma)$ satisfy
\begin{align}
\gamma\in\R, \,\, q\ge 2, \,\, r> 2,\,\, \frac2q+\frac {2d-1}{r}<\frac{2d-1}2,\,\, \mbox{and }\,\, \frac2q+\frac d{r}=\frac d2+\gamma,\label{Str-conditions}
\end{align}
moreover, when $\frac{4d+2}{2d-1}\le q\le \infty$, the penultimate inequality allows equality.
Then
$$
\big\||\nabla|^{\gamma} e^{it\Delta}g\big\|_{L^q_{t}L^r_x(\R\times\R^d)}\lesssim  \big\|g\big\|_{L^2(\R^d)}.
$$
Furthermore, let $F\in L^{\tilde q'}_{t}L^{\tilde r'}_x(\R^{d+1})$ be a radial function in $x$, then
$$
\Big\|\int_0^t e^{i(t-s)\Delta}F(s)\,ds\Big\|_{L^q_{t}L^r_x(\R^{d+1})}+
\Big\||\nabla|^{-\gamma}\int_0^t e^{i(t-s)\Delta}F(s)\,ds\Big\|_{L^\infty_tL^2_{x}(\R^{d+1})}\lesssim \|F\|_{L^{\tilde q'}_{t}L^{\tilde r'}_x(\R^{d+1})},
$$
where the triples  $(q,r,\gamma)$,   $(\tilde q,\tilde r,-\gamma)$ satisfy \eqref{Str-conditions}.
\end{lem}

\subsection{Incoming/outgoing operators and some basic properties} \label{sec:LE-Out-in}
In this subsection, we recall the definitions of the incoming and outgoing operators which were inspired by Tao \cite{Tao-2} and constructed in our previous work \cite{BDSW-II}, as well as some basic lemmas on their properties, which were proved in \cite{BDSW-II}. We  first introduce the deformed Fourier transform and its basic properties. For radial function $f\in \mathcal S(\R^d)$ (more general functions can be defined by density), we define the deformed Fourier transform,
\begin{align}
\mathcal F f(\rho)&=
\int_0^{+\infty}\!\!\int_{-\frac\pi2}^{\frac\pi2}e^{-2\pi i \rho r\sin \theta}\cos^{d-2} \theta  r^{\frac{3}{2}(d-1)-2} f(r) \,d\theta\,dr,\label{deformed-Fourier-radial}
\end{align}
which is the standard Fourier transform of $|x|^{\frac{d-1}{2}-2}f$ for radial $f$.
Then we have the inverse transform,
\begin{align}
f(r)&=r^{-\frac{d-1}{2}+2}\int_0^{+\infty}\!\!\int_{-\frac\pi2}^{\frac\pi2}e^{2\pi i \rho r\sin \theta}\cos^{d-2} \theta \rho^{d-1} \mathcal F f(\rho)\,d\theta d\rho.\label{inverse-deformed-Fourier-radial}
\end{align}

For convenience,  denote by
$$
J(r)=\int_0^{\frac\pi2}e^{2\pi i  r\sin \theta}\cos^{d-2} \theta\,d\theta.
$$
Then we have
\begin{align}
f(r)&=r^{-\frac{d-1}{2}+2}\int_0^{+\infty}\!\!\Big(J(\rho r)+J(-\rho r)\Big) \rho^{d-1} \mathcal F f(\rho)\,d\theta d\rho.\label{inverse-deformed-Fourier-radial}
\end{align}
Let
$$
K( r)=\chi_{\ge 1}( r)\Big[-\frac{1}{2\pi i r}-\frac{d-3}{(2\pi i r)^3}\Big], \quad \mbox{for } d=3,4,5.
$$
Then the form $J(r)-K(r)$ has some good properties as follows.
\begin{lem}\label{lem:J-K}
Let $d=3,4,5$, then there exist functions $a(r)$ and $\eta(\theta)$ satisfying
$$
a(r)=O\big(\langle r\rangle^{-5}\big);\qquad \eta(\theta)\in C^3\Big(\big[0,\frac \pi2\big]\Big),
$$
such that
\begin{align*}
J(r)-K(r)&=\int_0^{\frac\pi2} e^{2\pi i  r\sin \theta}\chi_{\ge \frac\pi6}(\theta)\cos^{d-2} \theta  \,d\theta+ a(r)\int_0^{\frac\pi2} e^{2\pi i  r\sin \theta}\eta(\theta)  \,d\theta.
\end{align*}
\end{lem}

Next, we define the incoming and outgoing decomposition in terms of the deformed Fourier transform as follows.
\begin{definition}\label{def:outgong-incoming}
Let  $f\in L^1_{loc}(\R^d)$ be a radial function.
We define the incoming component of $f$ as
$$
f_{in}(r)=r^{-\frac{d-1}{2}+2}\int_0^{+\infty}\!\!\Big(J(-\rho r)+K(\rho r)\Big) \rho^{d-1} \mathcal F f(\rho)\, d\rho;
$$
the outgoing component of $f$ as
$$
f_{out}(r)=r^{-\frac{d-1}{2}+2}\int_0^{+\infty}\!\!\Big(J(\rho r)-K(\rho r)\Big) \rho^{d-1} \mathcal F f(\rho)\, d\rho.
$$
\end{definition}
A very important  property of the outgoing/incoming component of function $f$ is that the following identity
$$
f=f_{in}+f_{out}
$$
holds, which is obtained directly from the definition above. Some further properties will be stated in the rest of this section.

The following lemma shows that if $f$ is supported outside of a ball, then $f_{out/in}$ is also almost supported outside of the ball.
\begin{lem}\label{lem:supportf+}
Let $\mu(d)=2$ if $d=3,4$, and $\mu(5)=3$.  Suppose that supp$f\subset \{x:|x|\geq1\}$, then
\begin{align*}
\big\|\chi_{\le \frac 14} (P_{\ge 1}f)_{out/in}\big\|_{H^{\mu(d)}(\R^d)}
\lesssim \|P_{2^{k}}f\|_{H^{-1}(\R^d)}.
\end{align*}
\end{lem}

We also need the boundedness of incoming/outgoing projection on $\dot H^s(\R^d)$.
\begin{lem}
\label{prop:bound-fpm-L2}
Suppose that $f\in L^2(\R^d)$, then for any $k\in \Z^+$ and $s\in [0,1]$,
\begin{align*}
\big\|\chi_{\ge \frac 14} \big(P_{2^{k}}f\big)_{out/in}\big\|_{\dot H^s(\R^d)}&\lesssim 2^{ks}\|f\|_{L^2(\R^d)}.
\end{align*}
Here the implicit constant is independent on $k$.
\end{lem}

The last one is the following simplified form of $f_{in/out}$, which follows from Lemma \ref{lem:J-K}.
\begin{lem}\label{prop:f++-highfreq}
Let $k$ be an integer.
Suppose that $f\in L^2(\R^d)$ with supp$f\subset \{x:|x|\ge 1\}$, then
\begin{align*}
\big(P_{2^{k}}f\big)_{out/in}(r)=&r^{-\frac{d-1}{2}+2}\int_0^{+\infty}\!\!\!\int_0^{\frac\pi2} e^{2\pi i  r\sin \theta}\chi_{\ge \frac\pi6}(\theta)\cos^{d-2} \theta  \,d\theta\\
&\qquad\qquad\cdot\chi_{2^{k-1}\le \cdot \le 2^{k+1}}(\rho)\mathcal F\big(P_{2^k} f\big)(\rho)\rho^{d-1}\,d\rho+\mathcal R\big(P_{2^k} f\big),
\end{align*}
with
\begin{align*}
\big\|\chi_{\ge \frac14}\mathcal R\big(P_{2^k} f\big)\big\|_{H^{\mu(d)}(\R^d)}\lesssim 2^{-k}\|f\|_{H^{-1}(\R^d)}.
\end{align*}
\end{lem}
\begin{proof}
We only sketch the proof here, one can see details in \cite{BDSW-II}.  It suffices to give the formula for $\big(P_{2^{k}}f\big)_{in}$ since the proof for $\big(P_{2^{k}}f\big)_{out}$ shares essentially the same procedures.

According to the support of $f$, we may write
$\mathcal F \big(P_{2^k}f\big)$ as the standard Fourier transform of $|x|^{\frac{d-1}{2}-2}P_{2^k}\chi_{\ge \frac12}f$. Hence, using the mismatch estimates in Lemma \ref{lem:mismatch}, we have
\begin{align*}
\big(P_{2^k}f\big)_{in}(r)=&r^{-\frac{d-1}{2}+2}\int_0^{+\infty}\!\!\Big(J(-\rho r)+K(\rho r)\Big) \rho^{d-1}\chi_{2^{k-1}\le \cdot \le 2^{k+1}}(\rho) \mathcal F \big(P_{2^k}f\big)(\rho)\, d\rho+h_k^1,
\end{align*}
with
\begin{align*}
\big\|h_k^1\big\|_{H^{\mu(d)}(\R^d)}\lesssim 2^{-k}\|f\|_{H^{-1}(\R^d)}.
\end{align*}
Moreover, by using  Lemma \ref{lem:J-K}, we further write
\begin{align*}
r^{-\frac{d-1}{2}+2}&\int_0^{+\infty}\!\!\Big(J(-\rho r)+K(\rho r)\Big) \rho^{d-1}\chi_{2^{k-1}\le \cdot \le 2^{k+1}}(\rho) \mathcal F \big(P_{2^k}f\big)(\rho)\, d\rho\\
=&r^{-\frac{d-1}{2}+2}\int_0^{+\infty}\!\!\!\int_0^{\frac\pi2} e^{2\pi i  r\sin \theta}\chi_{\ge \frac\pi6}(\theta)\cos^{d-2} \theta  \,d\theta
\>\chi_{2^{k-1}\le \cdot \le 2^{k+1}}(\rho)\mathcal F\big(P_{2^k} f\big)(\rho)\rho^{d-1}\,d\rho+h_k^2,
\end{align*}
with
\begin{align*}
\big\|\chi_{\ge \frac14}h_k^2\big\|_{H^{\mu(d)}(\R^d)}\lesssim 2^{-k}\|f\|_{H^{-1}(\R^d)}.
\end{align*}
Let $\mathcal R\big(P_{2^k} f\big)=h_k^1+h_k^2$ and then it satisfies the desired estimate. Hence we finish the proof.
\end{proof}

\section{Estimates on the incoming/outgoing linear flow}\label{sec:LE-Out-in}

In this section, we give the estimates on the linear flow when the initial data is the outgoing or incoming. Let $f$ be  a radial function  satisfying supp$f\subset \{x:|x|\geq1\}$. We abuse the notation and write $f=\chi_{\ge 1}f$ for simplicity. In the following, we focus on the estimates on the outgoing part, the estimates on the incoming part are similar.
Fixing an integer $k_0\ge 0$, we consider the estimates on $e^{it\Delta}\Big(P_{\ge 2^{k_0}}f\Big)_{out}$. For function $v$, we will use the notation $v_L=e^{it\Delta}v$ for short in the following.

Due to Lemmas \ref{lem:supportf+} and \ref{prop:f++-highfreq}, we may write
\begin{align*}
\Big(P_{\ge 2^{k_0}}f\Big)_{out,L}=\Big(P_{\ge 2^{k_0}}f\Big)_{out,L}^I+\Big(P_{\ge 2^{k_0}}f\Big)_{out,L}^{II},
\end{align*}
where
\begin{align}
\Big(P_{\ge 2^{k_0}}f\Big)_{out,L}^I=& \Big(\chi_{\le \frac 14} (P_{\ge 2^{k_0}}f)_{out}\Big)_{L}+\sum\limits_{k= k_0}^\infty \chi_{\le c(1+2^kt)}\cdot\big(\chi_{\ge \frac14}
\mathcal R\big(P_{2^k} f\big)\big)_{L}\notag\\
&+\sum\limits_{k= k_0}^\infty \chi_{\le c(1+2^kt)}\cdot e^{it\Delta}\Big(\chi_{\ge \frac14}(r)\>r^{-\frac{d-1}{2}+2}\int_0^{+\infty}\!\!\!\int_0^{\frac\pi2} e^{2\pi i \rho r\sin \theta}\notag\\
&\qquad\cdot\chi_{\ge \frac\pi6}(\theta)\cos^{d-2} \theta  \,d\theta\cdot\chi_{2^{k-1}\le \cdot \le 2^{k+1}}(\rho)\mathcal F\big(P_{2^k} f\big)(\rho)\rho^{d-1}\,d\rho\Big);\label{16.49}
\end{align}
and
\begin{align}
\Big(P_{\ge 2^{k_0}}f\Big)_{out,L}^{II}=&\sum\limits_{k= k_0}^\infty \chi_{\ge c(1+2^kt)}\cdot \Big(\chi_{\ge \frac 14} (P_{2^{k}}f)_{out}\Big)_{L}.\label{16.50}
\end{align}
Here $c$ is small positive constant.
Then the estimates on $\Big(P_{\ge 2^{k_0}}f\Big)_{out,L}^I$ and $\Big(P_{\ge 2^{k_0}}f\Big)_{out,L}^{II}$ are included in the following two lemmas.
\begin{lem}\label{lem:Part-I}
Let $d=3,4,5$, $s\in [0,\mu(d)], q\ge2,r\ge2$ and $\frac2q+\frac dr= \frac d2$. Then the following estimates hold,
\begin{align*}
\Big\||\nabla|^s\Big(P_{\ge 2^{k_0}}f\Big)_{out,L}^I\Big\|_{L^q_tL^r_x(\R^+\times\R^d)}
\lesssim
\|f\|_{H^{-1}(\R^d)}.
\end{align*}
\end{lem}
\begin{proof}
We shall consider the estimates on the three pieces in \eqref{16.49}.
For the first two pieces, by using Lemmas \ref{lem:supportf+}, \ref{prop:f++-highfreq} and \ref{lem:strichartz},
we have
\begin{align}
\Big\||\nabla|^s\Big(\chi_{\le \frac 14} (P_{\ge 2^{k_0}}f)\Big)_{out,L}\Big\|_{L^q_tL^r_x(\R^+\times\R^d)}
\lesssim \|f\|_{H^{-1}(\R^d)}.\label{11.59}
\end{align}
and
\begin{align}
\Big\||\nabla|^s\sum\limits_{k= k_0}^\infty \chi_{\le c(1+2^kt)}\cdot\big(\chi_{\ge \frac14}
\mathcal R\big(P_{2^k} f\big)\big)_{L}\Big\|_{L^q_tL^r_x(\R^+\times\R^d)}
\lesssim \|f\|_{H^{-1}(\R^d)}.\label{11.59-II}
\end{align}
It remains to  consider the estimates for the third piece in \eqref{16.49}. To do this, we first claim that for $j= 0,1,2,3$,
\begin{align}
\Big||\nabla|^j\Big[\chi_{\le c(1+2^kt)}\cdot e^{it\Delta}\Big(\chi_{\ge \frac14}(r)\>r^{-\frac{d-1}{2}+2}e^{2\pi i \rho r\sin \theta}\Big)\Big]\Big|
\lesssim \langle t\rangle^{-\frac d2}\rho^{-10}.\label{16.47}
\end{align}
To prove the claim (\ref{16.47}), we first use the formula in Lemma \ref{lem:formula-St} and write
\begin{align*}
e^{it\Delta}\Big(\chi_{\ge \frac14}(r)\>r^{-\frac{d-1}{2}+2}& e^{2\pi i \rho r\sin \theta}\Big)(x)
=\frac{C}{t^{\frac d2}}\int_{\R^d}e^{i\frac{|x-y|^2}{4t}+2\pi i \rho|y|\sin \theta}\chi_{\ge \frac14}(y)|y|^{-\frac{d-1}{2}+2} \,dy\\
=&\frac{C}{t^{\frac d2}}e^{i\frac{|x|^2}{4t}} \int_{\R^d}e^{-i\frac{x\cdot y}{2t}+i\frac{|y|^2}{4t}+2\pi i \rho|y|\sin \theta}\chi_{\ge \frac14}(y)|y|^{-\frac{d-1}{2}+2}\,dy\notag\\
=& \frac{C}{t^{\frac d2}}e^{i\frac{|x|^2}{4t}}\int_{|\omega|=1}\int_0^{+\infty} e^{i\phi(r)}\chi_{\ge \frac14}(r)r^{\frac{d-1}{2}+2}\,dr d\omega,
\end{align*}
where $C\in \C$ may vary line to line,
and $\phi(r)=-\frac{x\cdot \omega}{2t}r+\frac{r^2}{4t}+2\pi \rho r\sin \theta$. Then it is easy to see that
\begin{align}\label{15.35}
\phi'(r)=-\frac{x\cdot \omega}{2t}+\frac{r}{2t}+2\pi\rho\sin \theta,\quad \phi''(r)=\frac1{2t},\quad \mbox{and }\quad \phi^{(j)}(r)=0,\quad  j\ge 3.
\end{align}
Note that when $|x|\le c(\frac12+2^kt)$, $\rho\sim  2^k$, $r\ge \frac18$ and $\sin \theta\ge\frac14$, by choosing $c$ small enough, we have
\begin{align}\label{15.34}
\phi'(r)\ge  \frac14\Big(\frac{r}{t}+\pi\rho\Big).
\end{align}
Then using the formula,
\begin{align*}
e^{i\phi(r)}=\frac1{i\phi'(r)} \partial_r\big(e^{i\phi(r)}\big),
\end{align*}
and integrating by parts $K$ times, we have that for some $c_K\in\C$,
\begin{align}
\chi_{\le c(1+2^kt)}& \cdot e^{it\Delta}\Big(\chi_{\ge \frac14}(r)\>r^{-\frac{d-1}{2}+2}e^{2\pi i \rho r\sin \theta}\Big)(x)
\notag\\
&=
\chi_{\le c(1+2^kt)}\cdot\frac{c_K}{t^{\frac d2}}e^{i\frac{|x|^2}{4t}}\int_{|\omega|=1}\int_0^{+\infty}\!\!\! e^{i\phi(r)} \partial_r\Big(\frac{1}{\phi'(r)}\partial_r\Big)^{K-1}\Big[\frac{1}{\phi'(r)}r^{\frac{d-1}{2}+2}\chi_{\ge \frac14}(r)\Big]\,dr d\omega.\label{16.37}
\end{align}
Notice that it follows from Lemma \ref{lem:muli-Lei-formula} that
\begin{align*}
\partial_r\Big(\frac{1}{\phi'(r)}&\partial_r\Big)^{K-1}\Big[\frac{1}{\phi'(r)}r^{\frac{d-1}{2}+2}\chi_{\ge \frac14}(r)\Big]
\\
&=\sum\limits_{\begin{subarray}{c}
l_1,\cdots,l_K\in\R,l'\in\R;\\
l_j\le
j;l_1+\cdots+l_K+l'=K
\end{subarray}}
\!\!\!C_{l_1,\cdots,l_K,l'}\partial_r^{l_1}\Big(\frac{1}{\phi'(r)}\Big)\cdots
\partial_r^{l_K}\Big(\frac{1}{\phi'(r)}\Big)\>\partial_r^{l'}\Big[r^{d-1-\beta}\chi_{\ge \frac14}(r)\Big],
\end{align*}
which combined with \eqref{15.35} and \eqref{15.34} implies
$$
\Big|\partial_r\Big(\frac{1}{\phi'(r)}\partial_r\Big)^{K-1}\Big[\frac{1}{\phi'(r)}r^{\frac{d-1}{2}+2}\chi_{\ge \frac14}(r)\Big]\Big|
\lesssim t^{\frac d2}\langle t\rangle^{-\frac d2}\rho^{-10}\chi_{\gtrsim \frac14}(r).
$$
Then inserting the above estimate into \eqref{16.37}, we obtain (\ref{16.47}) for $j=0$,
\begin{align*}
\Big|\chi_{\le c(1+2^kt)}\cdot e^{it\Delta}\Big(\chi_{\ge \frac14}(r)\>r^{-\frac{d-1}{2}+2}e^{2\pi i \rho r\sin \theta}\Big)\Big|
\lesssim \langle t\rangle^{-\frac d2}\rho^{-15}.
\end{align*}
The estimates on the $j$-th derivative in \eqref{16.47} share similar arguments, since when the derivatives hit $\chi_{\le c(1+2^kt)}$ and $\chi_{\ge \frac14}(r)\>r^{-\frac{d-1}{2}+2}$, the estimates would become better, and when the derivatives hit $e^{2\pi i \rho r\sin \theta}$, it only increases the power of $\rho$. Hence we finish the proof of \eqref{16.47}.

It  follows from H\"older's inequality and \eqref{16.47} that
\begin{align*}
\Big||\nabla|^j&\Big[\sum\limits_{k= k_0}^\infty \chi_{\le c(1+2^kt)}\cdot e^{it\Delta}\Big(\chi_{\ge \frac14}(r)\>r^{-\frac{d-1}{2}+2}
\int_0^{+\infty}\!\!\!\int_0^{\frac\pi2} e^{2\pi i \rho r\sin \theta}\chi_{\ge \frac\pi6}(\theta)\cos^{d-2} \theta  \,d\theta\notag\\
&\cdot \chi_{2^{k-1}\le \cdot \le 2^{k+1}}(\rho)\mathcal F\big(P_{2^k} f\big)(\rho)\rho^{d-1}\,d\rho\Big)\Big]\Big|
\lesssim
\langle t\rangle^{-\frac d2}2^{-5k} \big\|\chi_{\gtrsim 2^k}(\rho)\mathcal F \left(P_{2^k}\big(\chi_{\ge 1}f\big)\right)\big\|_{L^2_\rho}.
\end{align*}
We next prove that
\begin{align}
\big\|\chi_{\gtrsim 2^k}(\rho)\mathcal F \left(P_{2^k}\big(\chi_{\ge 1}f\big)\right)\big\|_{L^2_\rho}
 \lesssim
2^{-\frac{d-1}{2}k} \big\|P_{2^k}\big(\chi_{\ge 1}f\big)\big\|_{L^2(\R^d)}.\label{16.48}
\end{align}
To this end, notice that
$$
\mathcal F \left(P_{2^k}\big(\chi_{\ge 1}f\big)\right)=\mathscr F \big(|x|^{\frac{d-1}{2}-2}P_{2^k}\big(\chi_{\ge 1}f\big)\big),
$$
where $\mathscr F$ is the standard Fourier transformation, then  by Plancherel identity, H\"older's and Bernstein's inequalities and Lemma \ref{lem:mismatch}, we have
\begin{align*}
\big\|\chi_{\gtrsim 2^k}(\rho)&\mathcal F \left(P_{2^k}\big(\chi_{\ge 1}f\big)\right)\big\|_{L^2_\rho}
\lesssim
2^{-\frac{d-1}{2}k}\big\|\mathcal F \left(P_{2^k}\big(\chi_{\ge 1}f\big)\right)\big\|_{L^2_\xi(\R^d)}\\
&\lesssim
2^{-\frac{d-1}{2}k}\big\||x|^{\frac{d-1}{2}-2}P_{2^k}\big(\chi_{\ge 1}f\big)\big\|_{L^2(\R^d)}\\
&\lesssim
2^{-\frac{d-1}{2}k}\Big(\big\||x|^{\frac{d-1}{2}-2}\chi_{\le \frac12}P_{2^k}\big(\chi_{\ge 1}f\big)\big\|_{L^2(\R^d)}+\big\||x|^{\frac{d-1}{2}-2}\chi_{\ge \frac12}P_{2^k}\big(\chi_{\ge 1}f\big)\big\|_{L^2(\R^d)}\Big)\\
&\lesssim
2^{-\frac{d-1}{2}k}\Big(\big\|\chi_{\le \frac12}P_{2^k}\big(\chi_{\ge 1}f\big)\big\|_{L^\infty(\R^d)}+\big\|P_{2^k}\big(\chi_{\ge 1}f\big)\big\|_{L^2(\R^d)}\Big)\\
&\lesssim
2^{-\frac{d-1}{2}k}\big\|P_{2^k}\big(\chi_{\ge 1}f\big)\big\|_{L^2(\R^d)},
\end{align*}
which implies \eqref{16.48}.
Hence, we obtain
\begin{align*}
\Big||\nabla|^j&\Big[\sum\limits_{k= k_0}^\infty \chi_{\le c(1+2^kt)}\cdot e^{it\Delta}\Big(\chi_{\ge \frac14}(r)\>r^{-\frac{d-1}{2}+2}
\int_0^{+\infty}\!\!\!\int_0^{\frac\pi2} e^{2\pi i \rho r\sin \theta}\chi_{\ge \frac\pi6}(\theta)\cos^{d-2} \theta  \,d\theta\notag\\
&\cdot \chi_{2^{k-1}\le \cdot \le 2^{k+1}}(\rho)\mathcal F\big(P_{2^k} f\big)(\rho)\rho^{d-1}\,d\rho\Big)\Big]\Big|
\lesssim
\langle t\rangle^{-\frac d2}2^{-6k} \big\|P_{2^k}\big(\chi_{\ge 1}f\big)\big\|_{L^2(\R^d)}.
\end{align*}
By using H\"older's inequality again, we obtain
\begin{align*}
\Big\||\nabla|^j&\Big[\sum\limits_{k= k_0}^\infty \chi_{\le c(1+2^kt)}\cdot e^{it\Delta}\Big(\chi_{\ge \frac14}(r)\>r^{-\frac{d-1}{2}+2}
\int_0^{+\infty}\!\!\!\int_0^{\frac\pi2} e^{2\pi i \rho r\sin \theta}\chi_{\ge \frac\pi6}(\theta)\cos^{d-2} \theta  \,d\theta\notag\\
&\cdot \chi_{2^{k-1}\le \cdot \le 2^{k+1}}(\rho)\mathcal F\big(P_{2^k} f\big)(\rho)\rho^{d-1}\,d\rho\Big)\Big]\Big\|_{L^q_tL^r_x(\R^+\times\R^d)}
\lesssim
\|f\|_{H^{-1}(\R^d)}.
\end{align*}
Hence we finish the proof.
\end{proof}

We next consider the Strichartz estimates for $\Big(P_{\ge 2^{k_0}}f\Big)_{out,L}^{II}$ (see (\ref{16.50}) its definition).
Let  $\sigma_0$  be a positive constant satisfying
$$
\frac{1}{\sigma_0}=\frac{2d+1}{4d-2}-\frac2d.
$$
Then we have the following lemma.
\begin{lem}\label{lem:Part-II}
Let $d=3,4,5$ and $(q,r)$ be one of the following pairs
\begin{align}
(\infty,2), \quad (\infty, \frac{dp}{2}-),\quad (\infty-, \frac{dp}{2}),\quad (2p,dp),\quad (2,\sigma_0).\label{10.13}
\end{align}
Then we have
\begin{align}
\Big\|\Big(P_{\ge 2^{k_0}}f\Big)_{out,L}^{II}\Big\|_{L^q_tL^r_x(\R^+\times\R^d)}
\lesssim
\|f\|_{H^{\frac{s_c}{d}-}(\R^d)}. \label{11.17}
\end{align}
Furthermore, for any $\beta\in [0,1]$, it holds that
\begin{align}
\Big\||\nabla|^\beta\Big(P_{\ge 2^{k_0}}f\Big)_{out,L}^{II}\Big\|_{L^{2}_tL^{\sigma_{0}}_x(\R^+\times\R^d)}
\lesssim
\|f\|_{H^{\beta-\frac1{\sigma_0}+}(\R^d)}. \label{11.21}
\end{align}
\end{lem}
\begin{proof}
By Lemma \ref{lem:radial-Sob}, we have
\begin{align*}
\Big\|\Big(P_{\ge 2^{k_0}}f\Big)_{out,L}^{II}\Big\|_{L^q_tL^r_x(\R^+\times\R^d)}
\lesssim
&\sum\limits_{k= k_0}^\infty \Big\|\chi_{\ge c(1+2^kt)}\cdot \Big(\chi_{\ge \frac 14} (P_{2^{k}}f)_{out}\Big)_{L}\Big\|_{L^q_tL^r_x(\R^+\times\R^d)}\\
\lesssim
&\sum\limits_{k= k_0}^\infty \Big\|(1+2^kt)^{-\alpha} \Big\|\Big(\chi_{\ge \frac 14} (P_{2^{k}}f)_{out}\Big)_{L}\Big\|_{\dot H^s_x(\R^d)}\Big\|_{L^q_t(\R^+)}\\
\lesssim
&\sum\limits_{k= k_0}^\infty \big\|(1+2^kt)^{-\alpha}\big\|_{L^q_t(\R^+)}\cdot \big\|\chi_{\ge \frac 14} (P_{2^{k}}f)_{out}\big\|_{\dot H^s_x(\R^d)}.
\end{align*}
Here we set $\alpha=(d-1)(\frac12-\frac1r)$ and $s=\frac12-\frac1r$. Note that for any $(q,r)$ satisfies \eqref{10.13} and $q\alpha>1$,
$$
\big\|(1+2^kt)^{-\alpha}\big\|_{L^q_t(\R^+)}\lesssim 2^{-\frac kq}.
$$
Moreover, it follows from Lemma \ref{prop:bound-fpm-L2} that
$$
\big\|\chi_{\ge \frac 14} (P_{2^{k}}f)_{out}\big\|_{\dot H^s_x(\R^d)}
\lesssim 2^{ks}\|P_{2^{k}}f\|_{L^2(\R^d)}.
$$
Thus we have
\begin{align*}
\Big\|\Big(P_{\ge 2^{k_0}}f\Big)_{out,L}^{II}\Big\|_{L^q_tL^r_x(\R^+\times\R^d)}
\lesssim
\|f\|_{H^{\frac12-\frac1r-\frac1q+}(\R^d)}.
\end{align*}
In particular, when $(q,r)=(2,\sigma_0)$, it follows
\begin{align}
\Big\|\Big(P_{\ge 2^{k_0}}f\Big)_{out,L}^{II}\Big\|_{L^2_tL^{\sigma_0}_x(\R^+\times\R^d)}
\lesssim
\|f\|_{H^{-\frac1{\sigma_0}+}(\R^d)}.\label{15.39}
\end{align}
Notice that for the pairs $(q,r)$ in \eqref{10.13}, the maximal value of $\frac12-\frac1r-\frac1q$ is $\frac12-\frac2{dp}-=\frac{s_c}{d}-$, which  finishes the proof for  \eqref{11.17}.

The proof for \eqref{11.21} follows from similar method as above, we only sketch the proof. Write
\begin{align}
|\nabla|^\beta\Big[&\chi_{\ge c(1+2^kt)} e^{it\Delta}\Big(\chi_{\ge \frac14}\left(P_{2^k}f\right)_{out}\Big)\Big]\notag\\
 =&P_{\le1}|\nabla|^\beta\Big[\chi_{\ge c(1+2^kt)} e^{it\Delta}\Big(\chi_{\ge \frac14}\left(P_{2^k}f\right)_{out}\Big)\Big]\label{0.44-I}\\
&\quad+\chi_{\le c^2(1+2^kt)}P_{\ge1} |\nabla|^\beta\Big[\chi_{\ge c(1+2^kt)} e^{it\Delta}\Big(\chi_{\ge \frac14}\left(P_{2^k}f\right)_{out}\Big)\Big]\label{0.44-II}\\
&\qquad+\chi_{\ge c^2(1+2^kt)} P_{\ge1}|\nabla|^\beta\Big[\chi_{\ge c(1+2^kt)} e^{it\Delta}\Big(\chi_{\ge \frac14}\left(P_{2^k}f\right)_{out}\Big)\Big].\label{0.44-III}
\end{align}
By using the Bernstein inequality and \eqref{15.39}, we have
\begin{eqnarray}
\big\|\eqref{0.44-I}\big\|_{L^q_tL^r_x(\R^+\times\R^d)}\lesssim\|f\|_{H^{-\frac1{\sigma_0}+}(\R^d)}.\nonumber
\end{eqnarray}
Furthermore, it follows from mismatch estimate (see Lemma \ref{lem:mismatch}) that
\begin{eqnarray}
\big\|\eqref{0.44-II}\big\|_{L^q_tL^r_x(\R^+\times\R^d)}\lesssim\|f\|_{H^{-10}(\R^d)}.\nonumber
\end{eqnarray}
As for the term \eqref{0.44-III}, using the similar method as the one used in the proof of  \eqref{11.17} and Lemma \ref{lem:frac_Hs},   we obtain that it can be controlled by $\|f\|_{H^{\beta-\frac1{\sigma_0}+}(\R^d)}$. Hence we finish the proof.

\end{proof}

\section{Linear flow estimates for compactly supported functions}\label{sec:LE-CS}


In this section, we shall prove the Strichartz estimates for $e^{it\Delta}g$ with $g$ satisfying the assumptions in Theorem \ref{thm:main2}. The main result is stated as follows.
\begin{prop}\label{prop:lineares-thm2}
Suppose that the suitable smooth function $g$ satisfies that
$$
\mbox{supp\,} g\subset \{x:|x|\le 1\},
$$
then for any $\varrho\ge 2$, $\sigma\ge 2$ with $\frac2\varrho+\frac d\sigma\le \frac d2$, $\frac1\varrho+\frac{d-1}\sigma<\frac{d-1}{2}$,
\begin{align}
\big\|e^{it\Delta}g\big\|_{L^\varrho_tL^\sigma_x(\R\times \R^d)}
\lesssim
\|\langle \xi\rangle^{-\frac{d-2}{(d-1)\varrho}+}\hat g\|_{L^1(\R^d)}.
\end{align}
\end{prop}
The proof of the proposition shall be divided into  several parts. First of all, we show that the estimate holds for low frequency.
\begin{lem}\label{lem:low-fre}
Let $\varrho\ge 2$, $\sigma\ge 2$ with $\frac2\varrho+\frac d\sigma\le \frac d2$. Then
\begin{align*}
\Big\|e^{it\Delta}\big(\chi_{\le 1}P_{\le 1}g\big)\Big\|_{L^\varrho_tL^\sigma_x(\R\times \R^d)}
\lesssim
\|\chi_{\le 1}\hat g\|_{L^1}.
\end{align*}
\end{lem}
\begin{proof}
Let $s=\frac d2-\frac2\varrho-\frac d\sigma>0$. It follows from   Lemma \ref{lem:strichartz} that
\begin{align}
\Big\|e^{it\Delta}\big(\chi_{\le 1}P_{\le 1}g\big)\Big\|_{L^\varrho_tL^\sigma_x(\R\times \R^d)}
\lesssim
\big\|\chi_{\le 1}P_{\le 1}g\big\|_{\dot H^s(\R^d)}. \label{15.44}
\end{align}
We only need to show that
\begin{align}
\big\|\chi_{\le 1}P_{\le 1}g\big\|_{\dot H^s(\R^d)}\lesssim \|\chi_{\le 1}\hat g\|_{L^1(\R^d)}.\label{15.43}
\end{align}
Indeed,  it follows from interpolation, H\"older's and Berstein's inequalities that for $0\le s\le 2$,
\begin{align*}
\big\|\chi_{\le 1}P_{\le 1}g\big\|_{\dot H^s(\R^d)}
\lesssim &
\big\|\chi_{\le 1}P_{\le 1}g\big\|_{L^2(\R^d)}^{1-\frac s2}
\big\|\Delta\big(\chi_{\le 1}P_{\le 1}g\big)\big\|_{L^2(\R^d)}^{\frac s2}\\
\lesssim &
\big\|P_{\le 1}g\big\|_{L^2(|x|\lesssim1)}+
\big\|P_{\le 1}g\big\|_{L^2(|x|\lesssim1)}^{1-\frac s2}
\big\|\nabla P_{\le 1}g\big\|_{L^2(|x|\lesssim1)}^{\frac s2}\\
&\quad
+
\big\|P_{\le 1}g\big\|_{L^2(|x|\lesssim1)}^{1-\frac s2}
\big\|\Delta P_{\le 1}g\big\|_{L^2(|x|\lesssim1)}^{\frac s2}\\
\lesssim &
\big\|P_{\le 1}g\big\|_{L^\infty}.
\end{align*}
We can then finish the proof by applying the Young inequality.
\end{proof}

It remains the proof of the estimates for high frequency $P_{\ge 1}g$. By dyadic decomposition, we have
\begin{align*}
\chi_{\le 1}P_{\ge 1}g =\sum\limits_{N\ge 1} \chi_{\le 1}P_Ng.
\end{align*}
Hence in the following, we only consider the function with localized frequency.
The following lemma shows that for high frequency and short time, we can gain regularity for the free flow in suitable Sctrichartz norms.
\begin{lem}\label{lem:g-1-short}
Assume that $\varrho\ge 1$, $\sigma\ge 1$, $N\ge 1$ and $\gamma\in (0,1]$. Then
\begin{align*}
\left\|e^{it\Delta}\big(\chi_{\le 1}P_Ng\big)\right\|_{L^\varrho_tL^\sigma_x([-8N^{-\gamma},8N^{-\gamma}]\times \R^d)}
\lesssim
N^{-\frac\gamma \varrho}\|\chi_{\sim N}\hat g\|_{L^1}.
\end{align*}
\end{lem}
\begin{proof}
By H\"older's inequality, we only need to prove
$$
\left\|e^{it\Delta}\big(\chi_{\le 1}P_Ng\big)\right\|_{L^\infty_tL^\sigma_x([-8N^{-\gamma},8N^{-\gamma}]\times \R^d)}
\lesssim
\|\chi_{\sim N}\hat g\|_{L^1}.
$$
Moreover, it follows from interpolation that
\begin{align*}
\big\|e^{it\Delta}&\big(\chi_{\le 1}P_Ng\big)\big\|_{L^\infty_tL^\sigma_x([-8N^{-\gamma},8N^{-\gamma}]\times \R^d)}\\
\lesssim&
\left\|e^{it\Delta}\big(\chi_{\le 1}P_Ng\big)\right\|_{L^\infty_{t}L^1_x([-8N^{-\gamma},8N^{-\gamma}]\times \R^d)}^\frac1\sigma
\left\|e^{it\Delta}\big(\chi_{\le 1}P_Ng\big)\right\|_{L^\infty_{tx}([-8N^{-\gamma},8N^{-\gamma}]\times \R^d)}^{1-\frac1\sigma}.
\end{align*}
Hence, it reduces to give control the $L^1_x$-norm and $L^\infty_x$-norm of $e^{it\Delta}\big(\chi_{\le 1}P_Ng\big)$, which should be uniform in time.
To this end, we write
$$
1=\chi_0(x,y)+\sum\limits_{j=1}^\infty\chi_j(x,y),
$$
for which
$$
\chi_0(x,y)=\chi_{\le 1}\Big(\frac{y-x+2t\xi}{2|t|^\frac12}\Big);
$$
and for $j\ge1$,
$$
\chi_j(x,y)=\chi_{\le 1}\Big(\frac{y-x+2t\xi}{2^{j+1}|t|^\frac12}\Big)-\chi_{\le 1}\Big(\frac{y-x+2t\xi}{2^j|t|^\frac12}\Big).
$$
By using the same formula for $e^{it\Delta}$ as in Lemma \ref{lem:formula-St}, we write
\begin{align}
e^{it\Delta}\big(\chi_{\le 1}P_Ng\big) =\frac{1}{(4\pi it)^\frac d2}\int_{\R^d}\!\!\int_{\R^d} e^{\frac{i|x-y|^2}{4t}+iy\cdot \xi}\chi_{\le 1}(y)\chi_{N}(\xi)\hat g(\xi)\,dyd\xi.
\label{10.10}
\end{align}
We further split $e^{it\Delta}\big(\chi_{\le 1}P_Ng\big)$ into the following parts,
\begin{align}
&\frac{1}{(4\pi it)^\frac d2}\int_{\R^d}\!\!\int_{\R^d} e^{\frac{i|x-y|^2}{4t}+iy\cdot \xi}\chi_0(x,y)\chi_{\le 1}(y)\chi_{N}(\xi)\hat g(\xi)\,dyd\xi\label{10.11}\\
&\quad+\sum\limits_{j=1}^\infty\frac{1}{(4\pi it)^\frac d2}\int_{\R^d}\!\!\int_{\R^d} e^{\frac{i|x-y|^2}{4t}+iy\cdot \xi}\chi_j(x,y)\chi_{\le 1}(y)\chi_{N}(\xi)\hat g(\xi)\,dyd\xi\label{10.12}.
\end{align}

For \eqref{10.11}, by the definition of $\chi_0$, we have
$$
|y-x+2t\xi|\lesssim |t|^{\frac12}.
$$
Then it follows from the H\"older inequality that
\begin{align}
\|\eqref{10.11}\|_{L^\infty_{x}}
\lesssim &
\frac{1}{|t|^\frac d2}\int_{\R^d}\!\!\int_{|y-x+2t\xi|\lesssim |t|^{\frac12}}\chi_{\le 1}(y)\chi_{N}(\xi)|\hat g(\xi)|\,dyd\xi\notag\\
\lesssim &\|\chi_{N}\hat g\|_{L^1}.\label{3.14}
\end{align}
On the other hand, notice that
\begin{align*}
\|\eqref{10.11}\|_{L^1_x}
\lesssim &
\frac{1}{|t|^\frac d2}\int_{\R^d}\!\!\int_{\R^d}\|\chi_0(\cdot,y)\|_{L^1_x}\chi_{\le 1}(y)\chi_{N}(\xi)|\hat g(\xi)|\,dyd\xi.
\end{align*}
and
$$\|\chi_0(\cdot,y)\|_{L^1_x}
\lesssim
|t|^\frac d2,
$$
then we obtain
\begin{align*}
\|\eqref{10.11}\|_{L^1_x}
\lesssim &
\|\chi_{N}\hat g\|_{L^1},
\end{align*}
which implies that
\begin{align}
\|\eqref{10.11}\|_{L^\infty_tL^1_x}
\lesssim &
\|\chi_{N}\hat g\|_{L^1}. \label{3.14'}
\end{align}
Interpolation between \eqref{3.14} and \eqref{3.14'}, we obtain that  for any $\sigma\ge 1$,
\begin{align}
\|\eqref{10.11}\|_{L^\infty_tL^\sigma_x}
\lesssim &
\|\chi_{N}\hat g\|_{L^1}. \label{3.16}
\end{align}

We next consider the estimate for \eqref{10.12}. For simiplicity,
 denote by
 $$\eqref{10.12}_j=\frac{1}{(4\pi it)^\frac d2}\int_{\R^d}\!\!\int_{\R^d} e^{\frac{i|x-y|^2}{4t}+iy\cdot \xi}\chi_j(x,y)\chi_{\le 1}(y)\chi_{N}(\xi)\hat g(\xi)\,dyd\xi.$$
 For each $j$,  denote by$$
\phi(y)=-\frac{x\cdot y}{2t}+\frac{|y|^2}{4t}+y\cdot \xi,
$$
the phase function. It follows that
\begin{align}
\nabla \phi(y)=\frac{y-x}{2t}+\xi, \label{15.08}
\end{align}
and
\begin{align}
\partial_{jk}\phi(y)=\frac{\delta_{jk}}{2t},\quad
\partial_{jkh}\phi(y)=0,\  \  \mbox{ for any}\quad  j,k,h\in\{1,2,3\}.\label{10.52}
\end{align}
Then
\begin{align*}
\eqref{10.12}_j =\frac{e^{\frac{i|x|^2}{4t}}}{(4\pi it)^\frac d2}\int_{\R^d}\!\!\int_{\R^d} e^{i\phi(y)}\chi_j(x,y)\chi_{\le 1}(y)\chi_{N}(\xi)\hat g(\xi)\,dyd\xi.
\end{align*}
By using the identity
\begin{align}
e^{i\phi}=\nabla_y e^{i\phi}\cdot\frac{\nabla_y\phi}{i|\nabla_y\phi|^2},\label{9.29}
\end{align}
and integrating by parts $K$ times, we obtain that for some $C_K\in \C$,
\begin{align}
\eqref{10.12}_j
= &
\frac{C_K e^{\frac{i|x|^2}{4t}}}{t^\frac d2}\int_{\R^d}\!\!\int_{\R^d} e^{i\phi(y)}\nabla_y\cdot\Big(\frac{\nabla_y\phi}{i|\nabla_y\phi|^2}\nabla_y\Big)^{K-1}\cdot\Big(\frac{\nabla_y\phi}{i|\nabla_y\phi|^2}\chi_j(x,y)\chi_{\le 1}(y)\Big)\,dy\chi_{N}(\xi)\hat g(\xi)\,d\xi. \label{10.03}
\end{align}
We shall first prove that
\begin{align}
\Big|\nabla_y\cdot\Big(\frac{\nabla_y\phi}{i|\nabla_y\phi|^2}\nabla_y\Big)^{K-1}\cdot\Big(\frac{\nabla_y\phi}{i|\nabla_y\phi|^2}\chi_j(x,y)\chi_{\le 1}(y)\Big)\Big|
\lesssim 2^{-2jK}\chi_{\sim 1}\Big(\frac{y-x+2t\xi}{2^j|t|^\frac12}\Big)\chi_{\lesssim 1}(y).\label{10.36}
\end{align}
To this end,  by using  Lemma \ref{lem:muli-Lei-formula}, we may expand the left-hand side of \eqref{10.36} as
\begin{equation*}
\sum\limits_{\begin{subarray}{c}
l_1,\cdots,l_K\in\R^d,l'\in\R^d;\\
|l_j|\le
j;|l_1|+\cdots+|l_K|+|l'|=K
\end{subarray}}
C_{l_1,\cdots,l_K,l'}\partial_y^{l_1}\Big(\frac{\nabla_y\phi}{i|\nabla_y\phi|^2}\Big)\cdots
\partial_y^{l_K}\Big(\frac{\nabla_y\phi}{i|\nabla_y\phi|^2}\Big)\>\partial_y^{l'}\big(\chi_j(x,y)\chi_{\le 1}(y)\big).
\end{equation*}
It follows from \eqref{10.52} that
$$
\Big|\partial_y^{l_1}\Big(\frac{\nabla_y\phi}{i|\nabla_y\phi|^2}\Big)\Big|\lesssim \frac{\big|\big(\nabla_y^2\phi\big)^{|l_1|}\big|}{|\nabla_y\phi|^{|l_1|+1}}\lesssim \frac{1}{|t|^{|l_1|}|\nabla_y\phi|^{|l_1|+1}}.
$$
Hence we obtain
\begin{align}
\Big|\partial_y^{l_1}\Big(\frac{\nabla_y\phi}{i|\nabla_y\phi|^2}\Big)\cdots
\partial_y^{l_K}\Big(\frac{\nabla_y\phi}{i|\nabla_y\phi|^2}\Big)\Big|
\lesssim \frac{1}{|t|^{|l_1|+\cdots+|l_K|}|\nabla_y\phi|^{|l_1|+\cdots+|l_K|+K}}.\label{11.28}
\end{align}
By the definition of $\chi_j$, we have
$$
|y-x+2t\xi|\sim 2^j |t|^{\frac12},
$$
which combined with \eqref{15.08} imples
$$
|\nabla_y \phi|\gtrsim 2^j|t|^{-\frac12}.
$$
Moreover, one has
\begin{align*}
\Big|\partial_y^{l'}\big(\chi_j(x,y)\chi_{\le 1}(y)\big)\Big|\lesssim 2^{-j|l'|}|t|^{-\frac12|l'|}\big(1+2^{j|l'|}|t|^{\frac12|l'|}\big)\chi_{\sim 1}\Big(\frac{y-x+2t\xi}{2^j|t|^\frac12}\Big)\chi_{\lesssim 1}(y).
\end{align*}
Then combining the estimates above, we have that
\begin{align*}
\Big|\partial_y^{l_1}&\Big(\frac{\nabla_y\phi}{i|\nabla_y\phi|^2}\Big)\cdots
\partial_y^{l_K}\Big(\frac{\nabla_y\phi}{i|\nabla_y\phi|^2}\Big)\>\partial_y^{l'}\big(\chi_j(x,y)\chi_{\le 1}(y)\big)\Big|\\
&\lesssim
\frac{1+2^{j|l'|}|t|^{\frac12|l'|}}{2^{j|l'|}|t|^{|l_1|+\cdots+|l_K|+\frac12|l'|}|\nabla_y\phi|^{|l_1|+\cdots+|l_K|+K}}\cdot\chi_{\sim 1}\Big(\frac{y-x+2t\xi}{2^j|t|^\frac12}\Big)\chi_{\lesssim 1}(y)\\
&\quad \lesssim
\frac{1+2^{j|l'|}|t|^{\frac12|l'|}}{|t|^{\frac12(|l_1|+\cdots+|l_K|+|l'|-K)}2^{j(|l_1|+\cdots+|l_K|+|l'|+K\big)}}\cdot\chi_{\sim 1}\Big(\frac{y-x+2t\xi}{2^j|t|^\frac12}\Big)\chi_{\lesssim 1}(y).
\end{align*}
Since $|l_1|+\cdots+|l_K|+|l'|=K$, it is further controlled by
$$
2^{-K j}\chi_{\sim 1}\Big(\frac{y-x+2t\xi}{2^j|t|^\frac12}\Big)\chi_{\lesssim 1}(y),
$$
which gives \eqref{10.36}.

Now by using \eqref{10.36} and the fact
$$
|y-x+2t\xi|\sim 2^j|t|^{\frac12},
$$
we obtain that
\begin{align*}
\Big\|\nabla_y\cdot &\Big(\frac{\nabla_y\phi}{i|\nabla_y\phi|^2}\nabla_y\Big)^{K-1}\cdot\Big(\frac{\nabla_y\phi}{i|\nabla_y\phi|^2}\chi_j(x,y)\chi_{\le 1}(y)\Big)\Big\|_{L^1_yL^1_x}\\
&+\Big\|\nabla_y\cdot\Big(\frac{\nabla_y\phi}{i|\nabla_y\phi|^2}\nabla_y\Big)^{K-1}\cdot\Big(\frac{\nabla_y\phi}{i|\nabla_y\phi|^2}\chi_j(x,y)\chi_{\le 1}(y)\Big)\Big\|_{L^\infty_xL^1_y}
\lesssim 2^{-K j+3j}|t|^{\frac d2}.
\end{align*}
Then it follows that
\begin{align*}
\|\eqref{10.12}_j\|_{L^\infty_{tx}}+\|\eqref{10.12}_j\|_{L^\infty_{t}L^1_x}
\lesssim_K &
2^{3j-K j}\|\chi_{N}\hat g\|_{L^1}.
\end{align*}
Choosing $K=4$ and using interpolation, we get that for any $\sigma\ge 1$,
\begin{align*}
\|\eqref{10.12}_j\|_{L^\infty_tL^\sigma_x}
\lesssim &
2^{-j}\|\chi_{N}\hat g\|_{L^1},
\end{align*}
and then
\begin{align}
\|\eqref{10.12}\|_{L^\infty_tL^\sigma_x}
\lesssim &\|\chi_{N}\hat g\|_{L^1}. \label{3.15}
\end{align}

Combining \eqref{3.16}  and \eqref{3.15}, we obtain
$$
\left\|e^{it\Delta}\big(\chi_{\le 1}P_Ng\big)\right\|_{L^\infty_tL^\sigma_x([-\frac8{N^\gamma},\frac8{N^\gamma}]\times \R^d)}
\lesssim \|\chi_{N}\hat g\|_{L^1}.
$$
Hence, we finish the proof.
\end{proof}

The second lemma shows that the linear flow $e^{it\Delta}$ become outgoing after a short time.
\begin{lem}\label{lem:g-1-in}
Let $N\ge 1$. Then for any $\gamma\in (0,1]$ and any $t$ with $|t|\ge 8N^{-\gamma}$,
\begin{align*}
\Big|\chi_{\le \frac18N|t|}(x)\Big(e^{it\Delta}\big(\chi_{\le 1}P_Ng\big)\Big)(x)\Big|
\lesssim
\langle t\rangle^{-\frac d2}N^{-100}\big\|\chi_{N}\hat g\big\|_{L^1}.
\end{align*}
\end{lem}
\begin{proof}
As in the proof of the previous lemma,  denote by $\phi$
$$
\phi(y)=-\frac{x\cdot y}{2t}+\frac{|y|^2}{4t}+y\cdot \xi,
$$
the phase function. Then it follows from \eqref{10.10} that
\begin{align}
e^{it\Delta}\big(\chi_{\le 1}P_Ng\big) =\frac{1}{(4\pi it)^\frac d2}e^{\frac{i|x|^2}{4t}}\int_{\R^d}\!\!\int_{\R^d} e^{i\phi(y)}\chi_{\le 1}(y)\chi_{N}(\xi)\hat g(\xi)\,dyd\xi.
\label{formula:Normal-2}
\end{align}
Notice that when $|y|\le\frac32, |\xi|\ge \frac34 N$, $t\ge \frac 8N$ and $|x|\le \frac3{16}N|t|$,  we have
\begin{align}
|\nabla_y \phi(y)|\gtrsim N. \label{11.38}
\end{align}
By applying the identity \eqref{9.29}
and integrating by parts in $y$-variable $K$ times, we obtain
\begin{align}
e^{it\Delta}\big(\chi_{\le 1}P_Ng\big) =\frac{C_K}{|t|^\frac d2}&e^{\frac{i|x|^2}{4t}}\int_{\R^d}\!\!\int_{\R^d} e^{i\phi(y)}\nabla_y\cdot\Big(\frac{\nabla_y\phi}{i|\nabla_y\phi|^2}\nabla_y\Big)^{K-1}\notag\\
&\cdot\Big(\frac{\nabla_y\phi}{i|\nabla_y\phi|^2}\chi_{\le 1}(y)\Big)\chi_{N}(\xi)\hat g(\xi)\,dyd\xi,
\label{formula:Normal-3}
\end{align}
for some constant $C_K\in \C$. Note that by Lemma \ref{lem:muli-Lei-formula}, one has
\begin{align*}
\nabla_y\cdot&\Big(\frac{\nabla_y\phi}{i|\nabla_y\phi|^2}\nabla_y\Big)^{K-1}\cdot\Big(\frac{\nabla_y\phi}{i|\nabla_y\phi|^2}\chi_{\le 1}(y)\Big)\\
&=
\sum\limits_{\begin{subarray}{c}
l_1,\cdots,l_K\in\R^d,l'\in\R^d;\\
|l_j|\le
j;|l_1|+\cdots+|l_K|+|l'|=K
\end{subarray}}
C_{l_1,\cdots,l_K,l'}\partial_y^{l_1}\Big(\frac{\nabla_y\phi}{i|\nabla_y\phi|^2}\Big)\cdots
\partial_y^{l_K}\Big(\frac{\nabla_y\phi}{i|\nabla_y\phi|^2}\Big)\>\partial_y^{l'}\big(\chi_{\le 1}(y)\big).
\end{align*}
Then by using the estimates in \eqref{11.28}, \eqref{11.38} and the definition for $\chi_{\leq1}$, we  obtain that
\begin{align*}
\Big|\partial_y^{l_1}\Big(\frac{\nabla_y\phi}{i|\nabla_y\phi|^2}\Big)\cdots
\partial_y^{l_K}&\Big(\frac{\nabla_y\phi}{i|\nabla_y\phi|^2}\Big)\Big|
\lesssim \frac{1}{|t|^{|l_1|+\cdots+|l_K|}|\nabla_y\phi|^{|l_1|+\cdots+|l_K|+K}}\\
&\lesssim
\frac{1}{(|t|N)^{|l_1|+\cdots+|l_K|}N^K}
\lesssim N^{-K},
\end{align*}
and
$$
\partial_y^{l'}\big(\chi_{\le 1}(y)\big)\lesssim \chi_{\lesssim 1}(y),
$$
which imply that
$$
\Big|\nabla_y\cdot\Big(\frac{\nabla_y\phi}{i|\nabla_y\phi|^2}\nabla_y\Big)^{K-1}\cdot\Big(\frac{\nabla_y\phi}{i|\nabla_y\phi|^2}\chi_{\le 1}(y)\Big)\Big|
\lesssim N^{-K}\chi_{\lesssim 1}(y).
$$
Hence it follows \eqref{formula:Normal-3} and the estimate above that
\begin{align*}
\Big|e^{it\Delta}\big(\chi_{\le 1}P_Ng\big)\Big|
\lesssim_K &
 \frac{1}{|t|^\frac d2}N^{-K}\int_{\R^d}\!\!\int_{\R^d} \chi_{\le 1}(y)\,dy\chi_{N}(\xi)\big|\hat g(\xi)\big|\,d\xi\\
\lesssim_K &
\frac{1}{|t|^\frac d2}N^{-K} \big\|\chi_{N}\hat g\big\|_{L^1}.
\end{align*}
Moreover, since $|t|\gtrsim \frac1N$, we have
\begin{align*}
\Big|e^{it\Delta}\big(\chi_{\le 1}P_Ng\big)\Big|
\lesssim_K
\frac{1}{\langle t\rangle^\frac d2}N^{-K+\frac32}\big\|\chi_{N}\hat g\big\|_{L^1}.
\end{align*}
Now by choosing $K$ large enough, we could  obtain the desired estimate.
\end{proof}

We next prove the estimate for linear flow out the ball $B(0,\frac{1}{8}N|t|)$, which is a consequence of the radial Sobolev embedding theorem.
\begin{lem}\label{lem:g-1-out}
For any $\varrho,\sigma\ge 1$ satisfying  $\frac1\varrho+\frac{d-1}\sigma<\frac{d-1}{2}$, and any $\gamma\ge 0$,  the following estimate holds,
\begin{align*}
\Big\|\chi_{\ge \frac18N|t|}(x)e^{it\Delta}\big(\chi_{\le 1}P_Ng\big)\Big\|_{L^\varrho_tL^\sigma_x(\{|t|\ge 8N^{-\gamma}\}\times \R^d)}
\lesssim N^{\big[(d-1)\gamma-(d-2)\big](\frac12-\frac1\sigma)-\frac\gamma \varrho}\big\|\chi_{N}\hat g\big\|_{L^1}.
\end{align*}
\end{lem}
\begin{proof}
It follows from the radial Sobolev embedding (see Lemma \ref{lem:radial-Sob}) that
\begin{align}
\Big\|\chi_{\ge \frac18N|t|}(x)e^{it\Delta}\big(\chi_{\le 1}P_Ng\big)\Big\|_{L^\sigma_x(\R^d)}
\lesssim &
\big(N|t|\big)^{-(d-1)\big(\frac12-\frac1\sigma\big)}\left\||\nabla|^{\frac12-\frac1\sigma}e^{it\Delta}\big(\chi_{\le 1}P_Ng\big)\right\|_{L^2_x(\R^d)}\notag\\
\lesssim &
\big(N|t|\big)^{-(d-1)\big(\frac12-\frac1\sigma\big)}\left\||\nabla|^{\frac12-\frac1\sigma}\chi_{\le 1}P_Ng\right\|_{L^2_x(\R^d)}.\label{20.31}
\end{align}
Similarly to the proof of \eqref{15.43}, we have
$$
\left\||\nabla|^{\frac12-\frac1\sigma}\chi_{\le 1}P_Ng\right\|_{L^2_x(\R^d)}
\lesssim
N^{\frac12-\frac1\sigma}\left\|\chi_{\lesssim 1}P_Ng\right\|_{L^2_x(\R^d)}.
$$
Then it follows from H\"older's and Young's inequalities that
$$
\left\|\chi_{\lesssim 1}P_Ng\right\|_{L^2_x(\R^d)}
\lesssim
\left\|\chi_{\lesssim 1}P_Ng\right\|_{L^\infty_x(\R^d)}
\lesssim
\big\|\chi_{N}\hat g\big\|_{L^1(\R^d)},
$$
which leads to
$$
\left\||\nabla|^{\frac12-\frac1\sigma}\chi_{\le 1}P_Ng\right\|_{L^2_x(\R^d)}
\lesssim
N^{\frac12-\frac1\sigma}\big\|\chi_{N}\hat g\big\|_{L^1(\R^d)}.
$$
Combining the above three inequalities with \eqref{20.31}, we obtain that
\begin{align*}
\Big\|\chi_{\ge \frac18N|t|}(x)e^{it\Delta}\big(\chi_{\le 1}P_Ng\big)\Big\|_{L^\sigma_x(\R^d)}
\lesssim &
N^{-(d-2)\big(\frac12-\frac1\sigma\big)}|t|^{-(d-1)\big(\frac12-\frac1\sigma\big)}\big\|\chi_{N}\hat g\big\|_{L^1(\R^d)}.
\end{align*}
Note that
$$
\frac1\varrho<(d-1)\big(\frac12-\frac1\sigma\big),
$$
then it follows
\begin{align*}
\Big\|\chi_{\ge \frac18N|t|}(x)e^{it\Delta}\big(\chi_{\le 1}P_Ng\big)\Big\|_{L^\varrho_tL^\sigma_x(\{|t|\ge 8N^{-\gamma}\}\times \R^d)}
\lesssim N^{\big[(d-1)\gamma-(d-2)\big](\frac12-\frac1\sigma)-\frac\gamma \varrho}\big\|\chi_{N}\hat g\big\|_{L^1(\R^d)},
\end{align*}
which finishes the proof.
\end{proof}

By Lemmas \ref{lem:g-1-short}-\ref{lem:g-1-out}, we could prove the Strichartz estimates for each frequency of function $g$, which is stated as follows.
\begin{cor}\label{cor:g-1-high-long}
For any $\varrho\ge 1,\sigma\ge 2$ satisfying  $\frac1\varrho+\frac{d-1}\sigma<\frac{d-1}{2}$, we have\begin{align*}
\Big\|e^{it\Delta}\big(\chi_{\le 1}P_Ng\big)\Big\|_{L^\varrho_tL^\sigma_x(\R\times \R^d)}
\lesssim N^{-\frac{d-2}{(d-1)\varrho}}\big\|\chi_{N}\hat g\big\|_{L^1}.
\end{align*}
\end{cor}
\begin{proof}
First, we claim that for $\gamma\in (0,1]$,
\begin{align}
\Big\|e^{it\Delta}\big(\chi_{\le 1}P_Ng\big)\Big\|_{L^\varrho_tL^\sigma_x(\{|t|\ge 8N^{-\gamma}\}\times \R^d)}
\lesssim N^{\big[(d-1)\gamma-(d-2)\big](\frac12-\frac1\sigma)-\frac\gamma \varrho}\big\|\chi_{N}\hat g\big\|_{L^1}.\label{1.36}
\end{align}
Indeed, it follows from Lemma \ref{lem:g-1-in} that
\begin{align*}
\Big\|\chi_{\le \frac18N|t|}&(x)\Big(e^{it\Delta}\big(\chi_{\le 1}P_Ng\big)\Big)(x)\Big\|_{L^\varrho_tL^\sigma_x(\{|t|\ge 8N^{-\gamma}\}\times \R^d)}\\
&\lesssim
\Big\|\big(N|t|\big)^{\frac d\sigma}\chi_{\le \frac18N|t|}(x)\Big(e^{it\Delta}\big(\chi_{\le 1}P_Ng\big)\Big)(x)\Big\|_{L^\varrho_tL^\infty_x(\{|t|\ge 8N^{-\gamma}\}\times \R^d)}\\
&\qquad\lesssim
\Big\|\langle t\rangle^{-\frac d2+\frac d\sigma}N^{-100+\frac d\sigma}\Big\|_{L^\varrho_t(\{|t|\ge 8N^{-\gamma}\})}\big\|\chi_{N}\hat g\big\|_{L^1}
\lesssim
N^{-50}\big\|\chi_{N}\hat g\big\|_{L^1},
\end{align*}
which together with Lemma \ref{lem:g-1-out} gives \eqref{1.36}.

 Now  we recall from Lemma \ref{lem:g-1-short} that for $\gamma\in (0,1]$,
\begin{align}
\left\|e^{it\Delta}\big(\chi_{\le 1}P_Ng\big)\right\|_{L^\varrho_tL^\sigma_x([-8N^{-\gamma},8N^{-\gamma}]\times \R^d)}
\lesssim
N^{-\frac\gamma \varrho}\|\chi_{\sim N}\hat g\|_{L^1}.\label{1.37}
\end{align}
Choosing $\gamma=\frac{d-2}{d-1}$ and then
$$
\big[(d-1)\gamma-(d-2)\big](\frac12-\frac1\sigma)-\frac\gamma \varrho=-\frac\gamma \varrho,
$$
Hence by \eqref{1.36} and \eqref{1.37}, we obtain that
\begin{align*}
\Big\|e^{it\Delta}&\big(\chi_{\le 1}P_Ng\big)\Big\|_{L^\varrho_tL^\sigma_x(\{|t|\ge 8N^{-\gamma}\}\times \R^d)}+\left\|e^{it\Delta}\big(\chi_{\le 1}P_Ng\big)\right\|_{L^\varrho_tL^\sigma_x([-8N^{-\gamma},8N^{-\gamma}]\times \R^d)}\\
&\lesssim
N^{-\frac{d-2}{(d-1)\varrho}}\|\chi_{\sim N}\hat g\|_{L^1},
\end{align*}
which gives the desired estimates and  completes the proof of the corollary.
\end{proof}

\begin{proof}[Proof of Proposition \ref{prop:lineares-thm2}]
Since $g$ is supported in the unit ball, we abuse the notation and write $g=\chi_{\le 1}g$. Then by the Littlewood-Paley decomposition,
\begin{align*}
\big\|e^{it\Delta}g&\big\|_{L^\varrho_tL^\sigma_x(\R\times \R^d)}
\le
\Big\|e^{it\Delta}\big(\chi_{\le 1}P_{\le 1}g\big)\Big\|_{L^\varrho_tL^\sigma_x(\R\times \R^d)}+\sum\limits_{N\ge 1} \Big\|e^{it\Delta}\big(\chi_{\le 1}P_Ng\big)\Big\|_{L^\varrho_tL^\sigma_x(\R\times \R^d)}.
\end{align*}
By Lemma \ref{lem:low-fre} and Lemma \ref{cor:g-1-high-long}, it is further controlled by
\begin{align*}
\big\|\chi_{\le 1}\hat g\big\|_{L^1(\R^d)}+\sum\limits_{N\ge 1} N^{-\frac{d-2}{(d-1)\varrho}}\big\|\chi_{N}\hat g\big\|_{L^1(\R^d)}
\lesssim
\big\|\langle \xi\rangle^{-\frac{d-2}{(d-1)\varrho}+}\hat g\big\|_{L^1(\R^d)}.
\end{align*}
This finishes the proof of the
proposition.
\end{proof}

\section{The linear estimates for the infinite bubbles}

In this section, we introduce the Strichartz estimates for the free flow $e^{it\Delta}h$ with $h$ satisfying Assumption \ref{ass-f-2}. The following is the main result of this section.

\begin{prop}\label{prop:bubbles}
Let $s_c> 0$, $\gamma\in [0,s_c)$ and $q_\gamma=\frac{2(d+2)}{d-2(s_c-\gamma)}$. Assume that $h$ is the function satisfying Assumption \ref{ass-f-2}. Then
\begin{align*}
\big\||\nabla|^{\gamma}e^{it\Delta}h\big\|_{L^{q_\gamma}_{tx}(\R\times \R^d)}
\lesssim \epsilon^{\frac{s_c-\gamma}{d}}\|h\|_{\dot H^{s_c}(\R^d)}.
\end{align*}
\end{prop}
The proof of this proposition is followed by the lemma stated as follows, in which we obtain the smallness of the space-time estimates.
\begin{lem}\label{lem:small-spacetime}
Let $h_k$ be  a component of $h$ in Assumption \ref{ass-f-2}. Under the same assumptions as in Proposition \ref{prop:bubbles}. Then
\begin{align}
\big\||\nabla|^{\gamma}e^{it\Delta}h_k\big\|_{L^{q_\gamma}_{tx}(\R\times \R^d)}
\lesssim \epsilon^{\frac{s_c-\gamma}{d}}\|h_k\|_{\dot H^{s_c}(\R^d)}.\label{9.08}
\end{align}
\end{lem}
\begin{remark}
The weak result when $\epsilon=1$ is known from the Strichartz estimate. When $\epsilon$ is small, then the smallness for the space-time estimates is due to the weak topology of the space-time norm compared to the Sobolev norm, and the restriction on frequency.
\end{remark}
\begin{proof}[Proof of Lemma \ref{lem:small-spacetime}]
By the assumption of $h_{k}$ in  Assumption \ref{ass-f-2}, we have
$$
h_k=P_{\Omega_k}h_k,
$$
where  $P_{\Omega_k}$ is the projection satisfying
$$
\widehat{P_{\Omega_k}g}(\xi)=\chi_{\Omega_k}(\xi)\hat g(\xi)
$$
with
$$
\chi_{\Omega_k}(\xi)=\chi_{\le 1}\Big(\frac{\xi}{(1+\epsilon)2^k}\Big)-\chi_{\le 1}\Big(\frac{\xi}{2^k}\Big).
$$
Note that
\begin{align*}
\big\||\nabla|^{\gamma}e^{it\Delta}h_k\big\|_{L^{q_\gamma}_{tx}(\R\times \R^d)}
=\big\|P_{\Omega_k}e^{it\Delta}P_{\Omega_k}h_k\big\|_{L^{q_\gamma}_{tx}(\R\times \R^d)}.
\end{align*}
and
$$
P_{\Omega_k}g=\mathcal F^{-1}\big(\chi_{\Omega_k}\big)*g.
$$
Hence,  it follows from Young's inequality that
\begin{align}
\big\||\nabla|^{\gamma}e^{it\Delta}&P_{\Omega_k}h_k\big\|_{L^{q_\gamma}_{tx}(\R\times \R^d)}\notag\\
&=\big\|\mathcal F^{-1}\big(\chi_{\Omega_k}\big)*e^{it\Delta}|\nabla|^{\gamma}h_k\big\|_{L^{q_\gamma}_{tx}(\R\times \R^d)}\notag\\
&\qquad\lesssim
\big\|\mathcal F^{-1}\big(\chi_{\Omega_k}\big)\big\|_{L^{\tilde r}_x(\R^d)}\big\|e^{it\Delta}P_{\Omega_k}|\nabla|^{\gamma}h_k\big\|_{L^{q_\gamma}_{t}L^{r}_{x}(\R\times \R^d)},\label{0.49}
\end{align}
where $r, \tilde r$ are the parameters satisfying
$$
\frac1r=\frac12-\frac{2}{dq_\gamma},\quad \frac1{\tilde r}=1-\frac{s_c-\gamma}{d}.
$$
Note that $1< \tilde r< 2$ and
\begin{align*}
\mathcal F^{-1}\big(\chi_{\Omega_k}\big)
=\big((1+\epsilon)2^k\big)^d\mathcal F^{-1}\big(\chi_{\le 1}\big)\Big((1+\epsilon)2^kx\Big)-\big(2^{k}\big)^d\mathcal F^{-1}\big(\chi_{\le 1}\big)\Big(2^kx\Big).
\end{align*}
Then we have
\begin{align*}
\big\|\mathcal F^{-1}\big(\chi_{\Omega_k}\big)\big\|_{L^1(\R^d)}
&\le \Big\|\big((1+\epsilon)2^k\big)^d\mathcal F^{-1}\big(\chi_{\le 1}\big)\Big((1+\epsilon)2^kx\Big)\Big\|_{L^1(\R^d)}\\
&\qquad+\Big\|\big(2^k\big)^d\mathcal F^{-1}\big(\chi_{\le 1}\big)\Big(2^kx\Big)\Big\|_{L^1(\R^d)}\\
&\lesssim \big\|\mathcal F^{-1}\big(\chi_{\le 1}\big)\big\|_{L^1(\R^d)}
\lesssim 1;
\end{align*}
and
\begin{align*}
\big\|\mathcal F^{-1}\big(\chi_{\Omega_k}\big)\big\|_{L^2(\R^d)}
&=\big\|\chi_{\Omega_k}\big\|_{L^2(\R^d)}\\
&
\lesssim  \big\|\chi_{2^k\le\cdot\le (1+\epsilon)2^k}\big\|_{L^2(\R^d)}\lesssim \epsilon^\frac12 2^{\frac {dk}{2} }.
\end{align*}
which imply that
\begin{align}
\big\|&\mathcal F^{-1}\big(\chi_{\Omega_k}\big)\big\|_{L^{\tilde r}(\R^d)}
\lesssim \big\|\mathcal F^{-1}\big(\chi_{\Omega_k}\big)\big\|_{L^1(\R^d)}^{\frac2{\tilde r}-1}\big\|\mathcal F^{-1}\big(\chi_{\Omega_k}\big)\big\|_{L^2(\R^d)}^{2-\frac2{\tilde r}}
\lesssim \epsilon^{\frac{s_c-\gamma}{d}} 2^{(s_c-\gamma)k}. \label{ex-1}
\end{align}
On the other hand, by Lemma \ref{lem:radial-Str}, we have
\begin{align}
\big\||\nabla|^{\gamma}e^{it\Delta}P_{\Omega_k}h_k\big\|_{L^{q_\gamma}_{t}L^{r}_{x}(\R\times \R^d)}
\lesssim \||\nabla|^{\gamma}h_k\|_{L^2(\R^d)}\lesssim 2^{-(s_c-\gamma)}\|h_k\|_{\dot H^{s_c}(\R^d)}.\label{ex-2}
\end{align}
Then it follows from \eqref{0.49}-\eqref{ex-2} that
$$
\big\||\nabla|^{\gamma}e^{it\Delta}P_{\Omega_k}h_k\big\|_{L^{q_\gamma}_{tx}(\R\times \R^d)}
\lesssim \epsilon^{\frac{s_c-\gamma}{d}}\|h_k\|_{\dot H^{s_c}(\R^d)}.
$$
Hence we finish the proof.
\end{proof}

Now we turn to prove Proposition \ref{prop:bubbles}.
\begin{proof}[Proof of Proposition \ref{prop:bubbles}]
Using the Littlewood-Paley characterization of Lebesgue space, we have
\begin{align*}
\big\||\nabla|^{\gamma}e^{it\Delta}h\big\|_{L^{q_\gamma}_{tx}(\R\times \R^d)}
\lesssim
\Big\|\Big(\sum\limits_{k=0}^\infty\big||\nabla|^{\gamma}e^{it\Delta}h_k\big|^2\Big)^\frac12\Big\|_{L^{q_\gamma}_{tx}(\R\times \R^d)}\\
\lesssim
\Big(\sum\limits_{k=0}^\infty\big\||\nabla|^{\gamma}e^{it\Delta}h_k\big\|_{L^{q_\gamma}_{tx}(\R\times \R^d)}^2\Big)^\frac12.
\end{align*}
Applying Lemma \ref{lem:small-spacetime}, we  further bound it by
\begin{align*}
\epsilon^{\frac{s_c-\gamma}{d}}\Big(\sum\limits_{k=0}^\infty\|h_k\|_{\dot H^{s_c}(\R^d)}^2\Big)^\frac12
=\epsilon^{\frac{s_c-\gamma}{d}}\|h\|_{\dot H^{s_c}(\R^d)},
\end{align*}
which finishes the proof.
\end{proof}

\section{Proof of the main theorems}
\label{sec:proof_of_THM2}

In this section, we prove Theorems \ref{thm:main2} and \ref{thm:main3}.

\subsection{The proof of Theorem \ref{thm:main2}} \label{sec:Proof-THM1}
In this subsection, we mainly focus on the proof of Theorem \ref{thm:main2}, which is based on the estimates obtained in Section 3 and 4 and continuity argument.  Before proving the theorem, we first introduce the modified incoming/outgoing components of radial function, the appropriate working spaces, the homogeneous and inhomogeneous Strichartz estimates.

\subsubsection{Definitions of the modified incoming and outgoing components}
First of all, we define the modified incoming and outgoing components, say $f_+$ and $f_-$, of the function $f$.
For this purpose, we split the function $f$ as follows,
$$
f=\big(1-\chi_{\ge 1}\big)f+\big(1-P_{\ge 1}\big)\chi_{\ge 1}f+P_{\ge 1}\chi_{\ge 1}f.
$$
\begin{definition}\label{def:+-}
Let the radial function $f\in \mathcal S  (\R^d)$.
We define the modified outgoing component of $f$ as
$$
f_+=\frac12\big(1-\chi_{\ge 1}\big)f+\frac12\big(1-P_{\ge 1}\big)\chi_{\ge 1}f+\big(P_{\ge 1}\chi_{\ge 1}f\big)_{out};
$$
the modified incoming component of $f$ as
$$
f_-=\frac12\big(1-\chi_{\ge 1}\big)f+\frac12\big(1-P_{\ge 1}\big)\chi_{\ge 1}f+\big(P_{\ge 1}\chi_{\ge 1}f\big)_{in}.
$$
\end{definition}
We note that such definitions can be extended to more general functions. It follows from the definitions that
$$
f=f_++f_-.
$$
Hence, if $f\in H^{s}(\R^d)$, then at least one of $f_+$ and $f_-$ belongs to $H^{s}(\R^d)$.

\subsubsection{Definitions of working space}\label{5.12}
For simplicity, we introduce the following working spaces. We denote $X(I)$ for $I\subset \R^+$ to be the space under the norms
 \begin{align*}
\|h\|_{X(I)}:=&\|h\|_{L^\infty_t \dot H^{s_c}_x(I\times \R^d)}+\|h\|_{L^{2p}_tL^{dp}_x(I\times\R^{d})}+\|h\|_{L^{\frac{(d+2)p}{2}}_{tx}(I\times \R^d)}+\big\||\nabla|^{s_c}h\big\|_{L^{\frac{2(d+2)}{d}}_{tx}(I\times \R^d)}.
\end{align*}
Moreover, we denote $X_0(I)$ for $I\subset \R^+$ to be the space under the norms
\begin{align*}
\|h\|_{X_0(I)}:=&\|h\|_{L^{\infty-}_tL^{2+}_x(I\times \R^d)}+\|h\|_{L^{\frac{(d+2)p}{2}}_{t}L^{\frac{2d(d+2)p}{d(d+2)p-8}}_x(I\times \R^d)}+\big\|h\big\|_{L^{\frac{2(d+2)}{d}}_{tx}(I\times \R^d)}.
 \end{align*}
 Then we denote $Y(I)$ to be the space under the norms
 $$
 \|h\|_{Y(I)}:=\|h\|_{X(I)}+\|h\|_{X_0(I)}.
 $$

Moreover, we need the spaces, $Z_f(I)$ and $Z_g(I)$ which are tailor-made for $f_{out}$ and $g$ respectively. Before defining the spaces, let
$\epsilon$  be a fixed small positive constant and $\sigma_0,\gamma$  be the constants satisfying
$$
\frac{1}{\sigma_0}=\frac{2d+1}{4d-2}-\frac2d, \quad \mbox{and }\quad \gamma=\frac{d-1}{2d-1}-\epsilon.
$$
Also, we denote $\infty-=\frac1\epsilon$. Now we define $Z_f(I)$ to be the space under the norm
\begin{align*}
\big\|h\big\|_{Z_f(I)}:=&
\big\|h\big\|_{L^{2p}_tL^{dp}_x(I\times\R^{d})}+\sup\limits_{q\in[2,\frac{dp}{2}-\epsilon]}\big\|h\big\|_{L^{\infty}_tL^{q}_x(I\times\R^{d})}\\
&+\big\|h\big\|_{L^{\infty-}_tL^{\frac{dp}{2}}_x(I\times\R^{d})}+\big\|h\big\|_{\big(L^\infty_t H^{s_c}_x+L^2_t\dot W^{s_c-\gamma,\sigma_0}_x\big)(I\times\R^{d})};
\end{align*}
and we define $Z_g(I)$ to be the space under the norm
\begin{align*}
\big\|h&\big\|_{Z_g(I)}:=
\big\|h\big\|_{L^{\frac{2(d+2)}{d}}_{tx}(I\times \R^d)}+\big\|h\big\|_{L^{\frac{(d+2)p}{2}}_{tx}(I\times \R^d)}+\big\|h\big\|_{L^{2p}_tL^{dp}_x(I\times\R^{d})}\\
&+\sup\limits_{q\in[2+\epsilon,\infty]}\big\|h\big\|_{L^{\infty-}_tL^{q}_x(I\times\R^{d})}+\sup\limits_{q\in[\frac{2d}{d-2},\infty]}\big\||\nabla|^{s_c-\gamma}h\big\|_{L^2_tL^{q}_x(I\times\R^{d})}.
\end{align*}

\subsubsection{The linear estimates}

For convenience, we rewrite the initial data $u_0$ as
$$
u_0=\psi_0+\big(P_{\ge 1}\chi_{\ge 1}f\big)_{out}+g,
$$
where
$$
\psi_0=\frac12\chi_{\le 1}f+\frac12 P_{\le 1}\chi_{\ge 1}f.
$$
By \eqref{f-cond}, we claim that $\psi_0\in H^{s_c}(\R^d)$ with
\begin{align}
\big\|\psi_0\big\|_{H^{s_c}(\R^d)}\lesssim \delta_0. \label{21.23}
\end{align}
Indeed,
$$
\big\|\chi_{\le 1}f\big\|_{L^2(\R^d)}\lesssim \big\|\chi_{\le 1}f\big\|_{L^{p_c}(\R^d)}
\lesssim
\big\|\chi_{\le 1}f\big\|_{\dot H^{s_c}(\R^d)}.
$$
Hence, $\chi_{\le 1}f\in H^{s_c}(\R^d)$ and
$$
\big\|\chi_{\le 1}f\big\|_{H^{s_c}(\R^d)}\lesssim \delta_0.
$$
Moreover, since the index $s_1$ in \eqref{f-cond} is smaller than $s_c$, one has
$$
\big\|P_{\le 1}\chi_{\ge 1}f\big\|_{H^{s_c}(\R^d)}\lesssim \big\|\chi_{\ge 1}f\big\|_{H^{\varepsilon_1}(\R^d)}\lesssim \delta_0.
$$
Hence we obtain the claim \eqref{21.23}.

We still need some linear estimates related to $\big(P_{\ge 1}\chi_{\ge 1}f\big)_{out}$ and $g$. Denote by   $f_{out,L}$ the linear flow of
 \begin{equation*}
   \left\{ \aligned
    &i\partial_{t}\phi+\Delta \phi=0,
    \\
    &\phi(0,x)  =\big(P_{\ge 1}\chi_{\ge 1}f\big)_{out}(x).
   \endaligned
  \right.
 \end{equation*}
 that is, $f_{out,L}(t)=e^{it\Delta}\big(P_{\ge 1}\chi_{\ge 1}f\big)_{out}$. Similarly, we denote $g_{L}(t)=e^{it\Delta}g.$ Next we shall introduce the estimates of $f_{out,L}(t)$ and $g_{L}$ in spaces $Z_f(I)$ and $Z_g(I)$, respectively. Such estimates can be obtained by using the results in Subsections \ref{sec:LE-Out-in} and \ref{sec:LE-CS}.

\begin{lem}\label{lem:f-out-L}
Suppose that $f$ is a radial function verifying \eqref{f-cond}, then
$$
\big\|f_{out,L}\big\|_{Z_f(\R^+)}\lesssim \delta_0.
$$
\end{lem}
\begin{proof} Due to the decomposition in the beginning of Section \ref{sec:LE-Out-in}, we write
\begin{align*}
f_{out,L}=\Big(P_{\ge 1}\chi_{\ge 1}f\Big)_{out,L}^I+\Big(P_{\ge 1}\chi_{\ge 1}f\Big)_{out,L}^{II}.
\end{align*}

Let us first deal with the estimates for $\Big(P_{\ge 1}\chi_{\ge 1}f\Big)_{out,L}^I$. Notice that it follows from Lemma \ref{lem:Part-I} that
\begin{align*}
\Big\||\nabla|^{s_c}\Big(&P_{\ge 1}\chi_{\ge 1}f\Big)_{out,L}^I\Big\|_{L^{2p}_tL^{\frac{dp-2}{2dp}}_x(\R^+\times\R^{d})}\!\!+\Big\|\Big(P_{\ge 1}\chi_{\ge 1}f\Big)_{out,L}^I\Big\|_{L^\infty_t H^{s_c}_x(\R^+\times\R^{d})}
\lesssim \|P_{\ge 1}\chi_{\ge 1}f\|_{H^{-1}(\R^d)},
\end{align*}
which combined with interpolation implies
\begin{align}
\Big\|\Big(&P_{\ge 1}\chi_{\ge 1}f\Big)_{out,L}^I\Big\|_{L^{2p}_tL^{dp}_x(\R^+\times\R^{d})}+\sup\limits_{q\in[2,\frac{dp}{2}-\epsilon]}\Big\|\Big(P_{\ge 1}\chi_{\ge 1}f\Big)_{out,L}^I\Big\|_{L^{\infty}_tL^{q}_x(\R^+\times\R^{d})}\notag\\
&+\Big\|\Big(P_{\ge 1}\chi_{\ge 1}f\Big)_{out,L}^I\Big\|_{L^{\infty-}_tL^{\frac{dp}{2}}_x(\R^+\times\R^{d})}+\Big\|\Big(P_{\ge 1}\chi_{\ge 1}f\Big)_{out,L}^I\Big\|_{L^\infty_t H^{s_c}_x(\R^+\times\R^{d})}
\lesssim \|P_{\ge 1}\chi_{\ge 1}f\|_{H^{-1}(\R^d)}.\label{16.10}
\end{align}

Now we turn to the estimates for $\Big(P_{\ge 1}\chi_{\ge 1}f\Big)_{out,L}^{II}$. We need to prove that
\begin{align}
\Big\|\Big(&P_{\ge 1}\chi_{\ge 1}f\Big)_{out,L}^{II}\Big\|_{L^{2p}_tL^{dp}_x(\R^+\times\R^{d})}+\sup\limits_{q\in[2,\frac{dp}{2}-\epsilon]}\Big\|\Big(P_{\ge 1}\chi_{\ge 1}f\Big)_{out,L}^{II}\Big\|_{L^{\infty}_tL^{q}_x(\R^+\times\R^{d})}\notag\\
&+\Big\|\Big(P_{\ge 1}\chi_{\ge 1}f\Big)_{out,L}^{II}\Big\|_{L^{\infty-}_tL^{\frac{dp}{2}}_x(\R^+\times\R^{d})}+\Big\|\Big(P_{\ge 1}\chi_{\ge 1}f\Big)_{out,L}^{II}\Big\|_{L^2_t\dot W^{s_c-\gamma,\sigma_0}_x(\R^+\times\R^{d})}\notag\\
&\qquad\qquad\lesssim \|P_{\ge 1}\chi_{\ge 1}f\|_{H^{s_1}(\R^d)},\label{16.42}
\end{align}
where $s_{1}<s_{c}$ is the constant defined as in Theorem \ref{thm:main2}. By using Lemma \ref{lem:Part-II} and interpolation, we have
\begin{align*}
\Big\|\Big(P_{\ge 1}\chi_{\ge 1}f\Big)_{out,L}^{II}&\Big\|_{L^{2p}_tL^{dp}_x(\R^+\times\R^{d})}+\sup\limits_{q\in[2,\frac{dp}{2}-\epsilon]}\Big\|\Big(P_{\ge 1}\chi_{\ge 1}f\Big)_{out,L}^{II}\Big\|_{L^{\infty}_tL^{q}_x(\R^+\times\R^{d})}\notag\\
&+\Big\|\Big(P_{\ge 1}\chi_{\ge 1}f\Big)_{out,L}^{II}\Big\|_{L^{\infty-}_tL^{\frac{dp}{2}}_x(\R^+\times\R^{d})}
\lesssim \|P_{\ge 1}\chi_{\ge 1}f\|_{H^{\frac{s_c}{d}-}(\R^d)}
\end{align*}
and
$$
\Big\|\Big(P_{\ge 1}\chi_{\ge 1}f\Big)_{out,L}^{II}\Big\|_{L^2_t\dot W^{s_c-\gamma,\sigma_0}_x(\R^+\times\R^{d})}
\lesssim \|P_{\ge 1}\chi_{\ge 1}f\|_{H^{s_c-\frac{4d-1}{4d-2}+\frac2d+}(\R^d)},
$$
which give \eqref{16.42}. Hence, we finish the proof.
\end{proof}

\begin{lem}\label{lem:g-L}
Suppose that $g$ is a radial function verifying \eqref{g-cond}, then
$$
\big\|g_{L}\big\|_{Z_g(\R)}\lesssim \delta_0.
$$
\end{lem}
\begin{proof}
Let $(\varrho,\sigma)$ be one  of the following pairs
$$
\big(\frac{2(d+2)}{d},\frac{2(d+2)}{d}\big),\quad \big(\frac{(d+2)p}{2},\frac{(d+2)p}{2}\big),\quad (2p,dp),\quad (\infty-,q)
$$
with $q\in (2,\infty]$. It is easy to see that every pair above satisfies
$\frac2\varrho+\frac d\sigma\le \frac d2$ and $\frac1\varrho+\frac{d-1}\sigma<\frac{d-1}{2}$. Then by  Proposition \ref{prop:lineares-thm2}, for the constant $\varepsilon_{0}$ in $s_{2}$ (see definition above Theorem \ref{thm:main2}), we have that
\begin{align*}
\big\|g_L&\big\|_{L^{\frac{2(d+2)}{d}}_{tx}(\R\times \R^d)}+\big\|g_L\big\|_{L^{\frac{(d+2)p}{2}}_{tx}(\R\times \R^d)}+\big\|g_L\big\|_{L^{2p}_tL^{dp}_x(\R\times\R^{d})}\\
&+\sup\limits_{q\in[2+\epsilon,\infty]}\big\|g_L\big\|_{L^{\infty-}_tL^{q}_x(\R\times\R^{d})}\lesssim \|\langle\xi\rangle^{-\varepsilon_0}\hat g\|_{L^1(\R^d)};
\end{align*}
and for any $q\ge \frac{2d}{d-2}$,
\begin{align*}
\big\||\nabla|^{s_c-\gamma}g_L\big\|_{L^2_tL^{q}_x(\R\times\R^{d})}\lesssim
\big\|\langle \xi\rangle^{s_c-\gamma-\frac{d-2}{2(d-1)}+}\hat g\big\|_{L^1(\R^d)}.
\end{align*}
Let the constant $\varepsilon_0$ can also be chosen such that
$$
s_c-\gamma-\frac{d-2}{2(d-1)}+=s_c-\frac{d-2}{2(d-1)}-\frac{d-1}{2d-1}+\varepsilon_0.
$$
Then by assumptions on $g$ (see \eqref{g-cond}), we finish the proof.
\end{proof}

\subsubsection{The proof of Theorem \ref{thm:main2}}

Denote by $\psi=u-f_{out,L}-g_L$, then $\psi$ obeys the following equation,
 \begin{equation*}
   \left\{ \aligned
    &i\partial_{t}\psi+\Delta \psi=|u|^pu,
    \\
    &\psi(0,x)  =\psi_0(x).
   \endaligned
  \right.
 \end{equation*}
It follows from the Duhamel formula that
\begin{align}
\psi(t)=e^{it\Delta}\psi_0+\int_0^{t}e^{i(t-s)\Delta}\big(|u|^pu\big)\,ds.\label{Duh-formu}
\end{align}
By using the standard Strichartz estimate given in Lemma \ref{lem:strichartz} and \eqref{21.23}, we have
\begin{align}
\big\|e^{it\Delta}\psi_0\big\|_{Y(\R)}\lesssim \big\|\psi_0\big\|_{H^{s_c}(\R^d)}\lesssim \delta_0. \label{est:psi-X}
\end{align}
It remains to consider the estimates for nonlinear part.  The following lemma deals with the estimate of the nonlinear term in $X_0(I)$.
\begin{lem}
Assume that  $d=3,4,5$,   $p\ge \frac 4d$, $0\in I$ and $\psi\in X(I)$, then
\begin{align*}
\Big\|\int_0^{t}e^{i(t-s)\Delta}\big(|u|^pu\big)\,ds\Big\|_{X_0(I)}\lesssim  \|\psi\|_{X_0(I)}\|\psi\|_{X(I)}^{p}+\delta_0\|\psi\|_{X(I)}^{p}+\delta_0^{p}\|\psi\|_{X_0(I)}+\delta_0^{p+1}.
\end{align*}
\end{lem}
\begin{proof}
Notice that by using Lemmas \ref{lem:f-out-L} and \ref{lem:g-L}, we have
\begin{align}
\big\|f_{out,L}\big\|_{L^{\infty-}_tL^{2+}_x(\R^+\times \R^d)}+\big\|f_{out,L}\big\|_{L^{2p}_tL^{dp}_x(\R^+\times \R^d)}\lesssim \delta_0\label{est:f-out-X0}
\end{align}
and
\begin{align}
\big\|g_L\big\|_{L^{\infty-}_tL^{2+}_x(\R\times \R^d)}+\big\|g_L\big\|_{L^{2p}_tL^{dp}_x(\R\times \R^d)}\lesssim \delta_0. \label{est:g-X0}
\end{align}
Since $u=f_{out,L}+g_L+\psi$,  it follows from (\ref{est:f-out-X0}) and (\ref{est:g-X0}) that
\begin{align*}
\big\|u\big\|_{L^{\infty-}_tL^{2+}_x(I\times \R^d)}&\lesssim \big\|\psi\big\|_{L^{\infty-}_tL^{2+}_x(I\times \R^d)}+\big\|f_{out,L}\big\|_{L^{\infty-}_tL^{2+}_x(\R^+\times \R^d)}+\big\|g_L\big\|_{L^{\infty-}_tL^{2+}_x(\R\times \R^d)}\\
&\lesssim \|\psi\|_{X_0(I)}+\delta_0
\end{align*}
and
\begin{align*}
\big\|u\big\|_{L^{2p}_tL^{dp}_x(I\times \R^d)}&\lesssim \big\|\psi\big\|_{L^{2p}_tL^{dp}_x(I\times \R^d)}+\big\|f_{out,L}\big\|_{L^{2p}_tL^{dp}_x(\R^+\times \R^d)}+\big\|g_L\big\|_{L^{2p}_tL^{dp}_x(\R\times \R^d)}\\
&\lesssim \|\psi\|_{X(I)}+\delta_0.
\end{align*}
Then by applying Lemma \ref{lem:strichartz}, we have
\begin{align*}
\Big\|\int_0^{t}e^{i(t-s)\Delta}\big(|u|^pu\big)\,ds&\Big\|_{X_0(I)}
\lesssim
\big\||u|^pu\big\|_{L^{2-}_tL^{\frac{2d}{d+2}+}(I\times \R^d)}\\
&\lesssim
\big\|u\big\|_{L^{\infty-}_tL^{2+}_x(I\times \R^d)}\big\|u\big\|_{L^{2p}_tL^{dp}_x(I\times \R^d)}^p\\
&
\lesssim \|\psi\|_{X_0(I)}\|\psi\|_{X(I)}^{p}+\delta_0\|\psi\|_{X(I)}^{p}+\delta_0^{p}\|\psi\|_{X_0(I)}+\delta_0^{p+1}.
\end{align*}
which implies the desired estimate. Hence we finish the proof.
\end{proof}

We still need to estimate the nonlinear term in $X(I)$, which require to deal with  the high-order derivatives. By Lemma \ref{lem:Frac_Leibniz}, we expand the nonlinearity and write
\begin{align}
|\nabla|^s\big(|u|^pu\big)=O\big((|\psi|^p+|f_{out,L}|^p+|g_L|^p)(|\nabla|^s \psi+|\nabla|^s f_{out,L}+|\nabla|^s g_L)\big),\label{NL-expand}
\end{align}
for any $s\ge 0$. We shall consider each term separately.
Here we abuse the notations and denote $O(f_1f_2)$ to be the product of the functions that it holds in the sense of H\"older's inequality, that is,
$$
\big\|O(f_1f_2)\big\|_{L^{q}}\lesssim \big\|f_1\big\|_{L^{q_1}}\|f_2\big\|_{L^{q_2}}, \quad \mbox{ with }\quad \frac{1}{q}=\frac1{q_1}+\frac1{q_2},\,\, q,q_1,q_2\in (1,+\infty].
$$

To consider the estimates of nonlinear term in  $X(I)$, it follows from  the Strichartz estimates in radial case (see Lemma \ref{lem:radial-Str}) that
\begin{align}
\Big\|\int_0^{t}e^{i(t-s)\Delta}\big(|u|^pu\big)\,ds&\Big\|_{X(I)}
\lesssim
\big\||\nabla|^{s_c-\gamma}\big(|u|^pu\big)\big\|_{L^2_tL^{\frac{4d-2}{2d+1}-}_x(I\times\R^{d})},\label{16.53}
\end{align}
where $\gamma={\frac{d-1}{2d-1}-}$. The estimates of the right hand side of (\ref{16.53}) would be done in the following three lemmas.



\begin{lem}\label{lem:NL-est-1} Assume that $d=3,4,5$,  $p\ge \frac 4d$, $0\in I$ and $\psi\in X(I)$. Then
\begin{align*}
\left\|(|\psi|^p+|f_{out,L}|^p+|g_L|^p)|\nabla|^{s_c-\gamma}\psi\right\|_{L^2_tL^{\frac{4d-2}{2d+1}-}_x(I\times\R^{d})}\lesssim \delta_0^{p+1}+\|\psi\|_{X(I)}^{p+1}.
\end{align*}
\end{lem}
\begin{proof}
Notice that it follows from  H\"older's inequality that
\begin{align*}
\Big\|(|\psi|^p+&|f_{out,L}|^p+|g_L|^p)|\nabla|^{s_c-\gamma}\psi\Big\|_{L^2_tL^{\frac{4d-2}{2d+1}-}_x(I\times\R^{d})}\\
&\lesssim
\big\||\nabla|^{s_c-\gamma}\psi\big\|_{L^\infty_tL^{q_1}_x(I\times\R^{d})}
\big(\|\psi\|_{L^{2p}_tL^{dp}_x(I\times\R^{d})}^p+\|f_{out,L}\|_{L^{2p}_tL^{dp}_x(\R^{+}\times\R^{d})}^p+\|g_L\|_{L^{2p}_tL^{dp}_x(\R\times\R^{d})}^p\big),
\end{align*}
where $q_1$ is the parameter satisfying
$$
\frac{1}{q_1}=\frac12-\frac{d-1}{d(2d-1)}+
$$
and
$$
\big\||\nabla|^{s_c-\gamma}\psi\big\|_{L^{q_1}_x(\R^{d})}\lesssim \big\||\nabla|^{s_c}\psi\big\|_{L^2_x(\R^{d})}.
$$
Thus we have
\begin{align*}
\Big\|(|\psi|^p&+|f_{out,L}|^p+|g_{L}|^{p})|\nabla|^{s_c-\gamma}\psi\Big\|_{L^2_tL^{\frac{4d-2}{2d+1}-}_x(I\times\R^{d})}\\
&\lesssim
\big\||\nabla|^{s_c}\psi\big\|_{L^\infty_tL^2_x(I\times\R^{d})}\big(\|\psi\|_{L^{2p}_tL^{dp}_x(I\times\R^{d})}^p
+\|f_{out,L}\|_{L^{2p}_tL^{dp}_x(\R^{+}\times\R^{d})}^p+\|g_L\|_{L^{2p}_tL^{dp}_x(\R\times\R^{d})}^p\big)\\
&\lesssim
\|\psi\|_{X(I)}\big(\|\psi\|_{X(I)}^{p}+\|f_{out,L}\|_{L^{2p}_tL^{dp}_x(\R^{+}\times\R^{d})}^{p}+\|g_L\|_{L^{2p}_tL^{dp}_x(\R\times\R^{d})}^p\big).
\end{align*}
On the other hand, it follows from Lemma \ref{lem:f-out-L} and Lemma \ref{lem:g-L} that
\begin{align*}
\big\|f_{out,L}\big\|_{L^{2p}_tL^{dp}_x(\R^+\times\R^{d})}+\|g_L\|_{L^{2p}_tL^{dp}_x(\R\times\R^{d})}\lesssim \delta_0,
\end{align*}
Then we could obtain the desired estimate by the Cauchy-Schwartz inequality.
This gives the proof of the lemma.
\end{proof}

\begin{lem}\label{lem:NL-est-4} Assume that $d=3,4,5$,  $p\ge \frac 4d$, $0\in I$ and $\psi\in Y(I)$. Then
\begin{align*}
\left\|(|\psi|^p+|f_{out,L}|^p+|g_L|^p)|\nabla|^{s_c-\gamma}f_{out,L}\right\|_{L^2_tL^{\frac{4d-2}{2d+1}-}_x(I\times\R^{d})}\lesssim \delta_0^{p+1}+\delta_0\|\psi\|_{X(I)}^{p}.
\end{align*}
\end{lem}
\begin{proof}
According to Lemma \ref{lem:f-out-L}, we split $f_{out,L}$ into two parts,
$$
f_{out,L}=f_{out,L}^I+f_{out,L}^{II},
$$
with
\begin{align*}
&\big\|f_{out,L}^I\big\|_{L^\infty_t H^{s_c}_x(\R^{+}\times\R^d)}\lesssim \delta_0,
\end{align*}
and
\begin{align*}
\big\|f_{out,L}^{II}\big\|_{L^2_t\dot W^{s_c-\gamma,\sigma_0}_x(\R^{+}\times\R^{d})}\lesssim \delta_0.
\end{align*}

For  ``$f_{out,L}^{I}$", similarly to the proof of Lemma \ref{lem:NL-est-1}, we obtain that
\begin{align*}
\Big\|(|\psi|^p&+|f_{out,L}|^p+|g_L|^p)|\nabla|^{s_c-\gamma}f_{out,L}^I\Big\|_{L^2_tL^{\frac{4d-2}{2d+1}-}_x(I\times\R^{d})}\\
&\lesssim
\big\||\nabla|^{s_c}f_{out,L}^I\big\|_{L^\infty_tL^2_x(\R^+\times\R^{d})}\big(\|\psi\|_{L^{2p}_tL^{dp}_x(I\times\R^{d})}+\|f_{out,L}\|_{L^{2p}_tL^{dp}_x(\R^{+}\times\R^{d})}
+\|g_L\|_{L^{2p}_tL^{dp}_x(\R\times\R^{d})}^p\big)\\
&\lesssim
\delta_0\big(\|\psi\|_{X(I)}^{p}+\|f_{out,L}\|_{L^{2p}_tL^{dp}_x(\R^{+}\times\R^{d})}^{p}+\|g_L\|_{L^{2p}_tL^{dp}_x(\R\times\R^{d})}^p\big)
\lesssim \delta_0^{p+1}+\delta_0\|\psi\|_{X(I)}^{p}.
\end{align*}

As for ``$f_{out,L}^{II}$", it follows from H\"older's inequality that
\begin{align*}
\Big\|(|\psi|^p&+|f_{out,L}|^p+|g_L|^p)|\nabla|^{s_c-\gamma}f_{out,L}^{II}\Big\|_{L^2_tL^{\frac{4d-2}{2d+1}-}_x(I\times\R^{d})}\\
&\lesssim
\|\psi\|_{L^{\infty}_tL^{\frac{dp}{2}-}_x(I\times\R^{d})}^p\big\||\nabla|^{s_c-\gamma}f_{out,L}^{II}\big\|_{L^2_tL^{\sigma_0}_x(\R^{+}\times\R^{d})}
 \\
 &\ \ \ \ \ \ +\ \|g_L\|_{L^{\infty}_tL^{\frac{dp}{2}-}_x(\R\times\R^{d})}^p
\big\||\nabla|^{s_c-\gamma}f_{out,L}^{II}\big\|_{L^2_tL^{\sigma_0}_x(\R^{+}\times\R^{d})}\\
&\ \ \ \ \ \ \ \ \ +\ \|f_{out,L}\|_{L^{\infty}_tL^{\frac{dp}{2}-}_x(\R^{+}\times\R^{d})}^p
\big\||\nabla|^{s_c-\gamma}f_{out,L}^{II}\big\|_{L^2_tL^{\sigma_0}_x(\R^{+}\times\R^{d})}.
\end{align*}
Note that $\|\psi\|_{L^{\infty}_tL^{\frac{dp}{2}-}_x(I\times\R^{d})}\lesssim \|\psi\|_{L^{\infty}_t H^{s_c}_x(I\times\R^{d})}$.
Therefore, by using Lemma \ref{lem:f-out-L} and Lemma \ref{lem:g-L}, we get
\begin{align*}
\Big\|(|\psi|^p+&|f_{out,L}|^p+|g_L|^p)|\nabla|^{s_c-\gamma}f_{out,L}^{II}\Big\|_{L^2_tL^{\frac{4d-2}{2d+1}-}_x(I\times\R^{d})}
\lesssim \delta_0^{p+1}+\delta_0\|\psi\|_{X(I)}^{p},
\end{align*}
which combined with the estimate for $f_{out,L}^{I}$ implies the desired estimate.
\end{proof}

\begin{lem}\label{lem:NL-est-5} Assume that $d=3,4,5$,  $p\ge \frac 4d$, $0\in I$ and $\psi\in Y(I)$. Then
\begin{align*}
\left\|(|\psi|^p+|f_{out,L}|^p+|g_L|^p)|\nabla|^{s_c-\gamma}g_L\right\|_{L^2_tL^{\frac{4d-2}{2d+1}-}_x(I\times\R^{d})}\lesssim \delta_0^{p+1}+\delta_0\|\psi\|_{X(I)}^{p}.
\end{align*}
\end{lem}
\begin{proof}
By using H\"older's inequality, we have
\begin{align*}
&\big\|(|\psi|^p+|f_{out,L}|^p+|g_L|^p)|\nabla|^{s_c-\gamma}g_L\big\|_{L^{2}_tL^{\frac{4d-2}{2d+1}-}_x(I\times\R^{d})}\\
&\lesssim
\big\||\nabla|^{s_c-\gamma}g_L\big\|_{L^{2+}_tL^{q_3}_x(\R\times\R^{d})}
\Big(\|\psi\|_{L^{\infty-}_tL^{\frac{dp}{2}}_x(I\times\R^{d})}^p+\|f_{out,L}\|_{L^{\infty-}_tL^{\frac{dp}{2}}_x(\R^{+}\times\R^{d})}^p+\|g_L\|_{L^{\infty-}_tL^{\frac{dp}{2}}_x(\R\times\R^{d})}^p\Big)
\end{align*}
with
$$
q_3=\frac{d(4d-2)}{2d^2-7d+4}-.
$$
It is easy to see that $q_3> \frac{2d}{d-2}$ when $d\ge 3$. Then by Lemma \ref{lem:g-L} and interpolation, we obtain
$$
\big\||\nabla|^{s_c-\gamma}g_L\big\|_{L^{2+}_tL^{q_3}_x(\R\times\R^{d})}
+\|g_L\|_{L^{\infty-}_tL^{\frac{dp}{2}}_x(\R\times\R^{d})}
\lesssim \delta_0.
$$
Notice that it follows from Lemma \ref{lem:f-out-L} that
$$
\|f_{out,L}\|_{L^{\infty-}_tL^{\frac{dp}{2}}_x(\R^+\times\R^{d})}\lesssim \delta_0.
$$
Hence we have
\begin{align*}
\big\|(|\psi|^p+&|f_{out,L}|^p+|g_L|^p)|\nabla|^{s_c-\gamma}g_L\big\|_{L^2_tL^{\frac{4d-2}{2d+1}-}_x(I\times\R^{d})}
\lesssim \delta_0^{p+1}+\delta_0\|\psi\|_{X(I)}^{p},
\end{align*}
which finishes the proof.
\end{proof}

Now we turn to close the estimate in $Y(I)$. It follows from from the Duhamel formula \eqref{Duh-formu} that for any $I\subset \R^+$,
\begin{align*}
\big\|\psi\big\|_{Y(I)}\lesssim
\big\|e^{it\Delta}\psi_0\big\|_{Y(\R^+)}+\Big\|\int_0^{t}e^{i(t-s)\Delta}\big(|u|^pu\big)\,ds\Big\|_{X_0(I)}+
\Big\|\int_0^{t}e^{i(t-s)\Delta}\big(|u|^pu\big)\,ds\Big\|_{X(I)}.
\end{align*}
By using \eqref{21.23} and Lemma \ref{lem:NL-est-1}--Lemma \ref{lem:NL-est-5}, we obtain that
\begin{align*}
\big\|\psi\big\|_{Y(I)}\lesssim
\delta_0+\|\psi\|_{Y(I)}^{p+1}+\delta_0\|\psi\|_{Y(I)}^{p}+\delta_0^{p}\|\psi\|_{Y(I)}+\delta_0^{p+1},
\end{align*}
where the implicit constant is independent on $I$. Then by continuity argument, we prove the global existence of the solution $\psi$ and
\begin{align*}
\big\|\psi\big\|_{Y(\R^+)}\lesssim
\delta_0.
\end{align*}

We consider the scattering. Denote by
$$
\psi_{0+}= \psi_0+\int_0^{+\infty} e^{-is\Delta}\big(|u|^pu\big)\,ds.
$$
It follows
$$
\psi(t)-e^{it\Delta}\psi_{0+}=\int_t^\infty e^{i(t-s)\Delta}\big(|u|^2u\big)\,ds.
$$
Then using Lemma \ref{lem:strichartz} and the argument as above, we obtain that
\begin{align*}
\|\psi(t)-e^{it\Delta}&\psi_{0+}\|_{\dot H^{s_c}}
\lesssim  \|\psi\|_{Y([t,+\infty))}^{p+1}+\big(\big\|f_{out,L}\big\|_{Z_f([t,+\infty))}+\big\|g_L\big\|_{Z_g([t,+\infty))}\big)\|\psi\|_{Y([t,+\infty))}^{p}\\
&\quad +\big(\big\|f_{out,L}\big\|_{Z_f([t,+\infty))}+\big\|g_L\big\|_{Z_g([t,+\infty))}\big)^{p}\|\psi\|_{Y([t,+\infty))}\\
&\qquad +\big\|f_{out,L}\big\|_{Z_f([t,+\infty))}^{p+1}+\big\|g_L\big\|_{Z_g([t,+\infty))}^{p+1}.
\end{align*}
Since
$$
\big\|\psi\big\|_{Y(\R^+)}+\big\|f_{out,L}\big\|_{Z_f(\R^+)}+\big\|g_L\big\|_{Z_g(\R)}<+\infty,
$$
we have
$$
\big\|\psi\big\|_{Y([t,+\infty))}+\big\|f_{out,L}\big\|_{Z([t,+\infty))}+\big\|g_L\big\|_{Z_g([t,+\infty))}\to 0,\quad \mbox{as }\quad t\to +\infty.
$$
Therefore,
\begin{align*}
\|\psi(t)-e^{it\Delta}\psi_{0+}\|_{\dot H^{s_c}}\to 0,\quad \mbox{as}\quad t\to +\infty.
\end{align*}
Setting
$$
u_{0+}=\big(P_{\ge 1}\chi_{\ge 1}f\big)_{out}+g+\psi_{0+},
$$
we obtain \eqref{scattering2} and thus finish the proof of Theorem \ref{thm:main2}.

\vskip0.3cm

\subsection{Proof of Theorem \ref{thm:main3}}

We denote $h_{k,L}$ to be the linear flow of
 \begin{equation*}
   \left\{ \aligned
    &i\partial_{t}\phi+\Delta \phi=0,
    \\
    &\phi(0,x)  =h_k(x).
   \endaligned
  \right.
 \end{equation*}
 that is, $h_{k,L}(t)=e^{it\Delta}h_k$. Moreover,
 $$
 h_L=\sum\limits_{k=k_0}^{+\infty}h_{k,L}.
 $$
Then it follows from Proposition \ref{prop:bubbles} and \eqref{ass-f} that for $\gamma \in [0,s_c)$,
\begin{align}
\big\||\nabla|^\gamma h_L\big\|_{L^{q_\gamma}_{tx}(\R\times \R^d)}
\lesssim \epsilon^{\frac{s_c-\gamma}{d}}\|h\|_{\dot H^{s_c}}
\lesssim \epsilon^{\frac{s_c-\gamma}{d}-\alpha_0}.\label{17.20}
\end{align}

Now consider $\psi=u-h_L$, which  obeys the following equation,
 \begin{equation*}
   \left\{ \aligned
    &i\partial_{t}\psi+\Delta \psi=|u|^p u,
    \\
    &\psi(0,x)  =\psi_0.
   \endaligned
  \right.
 \end{equation*}
To solve the above equation, we first introduce the working space as follows.
We denote $X(I)$ for $I\subset \R$ to be the space under the norm
\begin{align*}
\|\psi\|_{X(I)}:=&\epsilon^{-\frac{s_c}{2d}}\sup\limits_{\gamma\in [0,s_c-\gamma_0]}\big\||\nabla|^{\gamma}\psi\big\|_{L^{q_\gamma}_{tx}(I\times \R^d)}+\big\||\nabla|^{s_c}\psi\big\|_{L^\infty_t L^2_x(I\times \R^d)}++\big\||\nabla|^{s_c}\psi\big\|_{L^{\frac{2(d+2)}{d}}_{tx}(I\times \R^d)},
\end{align*}
where $\gamma_0$ is a parameter decided later and
$$q_\gamma=\frac{2(d+2)}{d-2(s_c-\gamma)}.$$
It follows from the Duhamel formula that
\begin{align}
\psi(t)=e^{it\Delta}\psi_0+\int_0^{t}e^{i(t-s)\Delta}\big(|u|^pu\big)\,ds.\label{Duh-formu-II}
\end{align}
Then by using Lemma \ref{lem:strichartz}, Lemma \ref{lem:strichartz-inhomog} and \eqref{Duh-formu-II}, we have that for any $I\subset \R$ and $\gamma\in [0,s_c-\gamma_0]$,
\begin{align}
\big\||\nabla|^{\gamma}\psi\big\|_{L^{q_\gamma}_{tx}(I\times \R^d)}
\lesssim \big\|\psi_0\big\|_{\dot H^{s_c}(\R^d)}+\big\||\nabla|^{\gamma}(|u|^pu)\big\|_{L^{q_1'}_t L^{q_\gamma'}_x(I\times \R^d)},\label{17.34}
\end{align}
where $q$ satisfy
$$
\frac1{q_1}=\frac d2-\frac{d+1}{q_\gamma}=\frac{d}{2(d+2)}+(d+1)\varepsilon.
$$
Here and in the following, we denote $\varepsilon=s_c-\gamma$ and $\varepsilon_0=s_c-\gamma_0$ for short. Moreover, we choose $\varepsilon_0$ small enough such that
for any $\varepsilon\in [0,\varepsilon_0]$, $q_1>1$.
By using the H\"older inequality and the fractional Leibniz rule in Lemma \ref{lem:Frac_Leibniz}, we obtain
\begin{align}
\big\||\nabla|^{\gamma}(|u|^pu)\big\|_{L^{q_1'}_t L^{q_\gamma'}_x(I\times \R^d)}
\lesssim
\big\||\nabla|^{\gamma}u\big\|_{L^{q_\gamma}_{tx}(I\times \R^d)} \big\|u\big\|_{L^{q_2}_{t}L^{r_2}_x(I\times \R^d)}^p,
\label{15.59}
\end{align}
where $q_2,r_2$ satisfy
$$
\frac1{q_2}=\frac{2}{(d+2)p}-\frac {d}{p(d+2)} \varepsilon;\quad \frac1{r_2}=\frac{2}{(d+2)p}+\frac {2}{p(d+2)}\varepsilon.
$$
It follows from interpolation and the Sobolev embedding that
\begin{align*}
 \big\|u\big\|_{L^{q_2}_{t}L^{r_2}_x(I\times \R^d)}
 \lesssim & \big\|u\big\|_{L^{\infty}_{t}L^{\frac{dp}{2}}_x(I\times \R^d)}^\theta\big\|u\big\|_{L^{\frac{(d+2)p}{2}}_{tx}(I\times \R^d)}^{1-\theta}\\
  \lesssim &\big\|u\big\|_{L^{\infty}_{t}\dot H^{s_c}_x(I\times \R^d)}^\theta\big\|u\big\|_{L^{\frac{(d+2)p}{2}}_{tx}(I\times \R^d)}^{1-\theta},
\end{align*}
where
$
\theta=\frac{d(d+2)}{2}\varepsilon<1
$
(after choosing $\varepsilon_0$ small enough). Notice that by \eqref{ass-f} and \eqref{17.20}, one has
$$
\big\|u\big\|_{L^{\infty}_{t}\dot H^{s_c}_x(I\times \R^d)}
\lesssim \big\|h_L\big\|_{L^{\infty}_{t}\dot H^{s_c}_x(I\times \R^d)}+\big\|\psi\big\|_{L^{\infty}_{t}\dot H^{s_c}_x(I\times \R^d)}
\lesssim\epsilon^{-\alpha_0}+\big\|\psi\big\|_{X(I)}
$$
and
$$
\big\|u\big\|_{L^{\frac{(d+2)p}{2}}_{tx}(I\times \R^d)}
\lesssim \big\|h_L\big\|_{L^{\frac{(d+2)p}{2}}_{tx}(I\times \R^d)}+\big\|\psi\big\|_{L^{\frac{(d+2)p}{2}}_{tx}(I\times \R^d)}
\lesssim\epsilon^{\frac{s_c}{d}-\alpha_0}+\epsilon^{\frac{s_c}{2d}}\big\|\psi\big\|_{X(I)}.
$$
Then,
\begin{align*}
 \big\|u\big\|_{L^{q_2}_{t}L^{r_2}_x(I\times \R^d)}
 \lesssim & \epsilon^{\frac{s_c}{d}(1-\theta)-\alpha_0}+\epsilon^{-\alpha_0\theta+\frac{s_c}{2d}(1-\theta)}\big\|\psi\big\|_{X(I)}^{1-\theta}\\
& +\epsilon^{(\frac{s_c}{d}-\alpha_0)(1-\theta)}\big\|\psi\big\|_{X(I)}^{\theta}+\epsilon^{\frac{s_c}{2d}(1-\theta)}\big\|\psi\big\|_{X(I)},
\end{align*}
which combined with \eqref{17.20} and \eqref{15.59} gives
\begin{align}
\big\||\nabla|^{\gamma}(|u|^pu)\big\|_{L^{q_1'}_t L^{q_\gamma'}_x(I\times \R^d)}
\lesssim& \epsilon^{\frac{s_c}{2d}}\big\|\psi\big\|_{X(I)}\cdot \Big(\epsilon^{\frac{s_c}{d}(1-\theta)-\alpha_0}+\epsilon^{-\alpha_0\theta+\frac{s_c}{2d}(1-\theta)}\big\|\psi\big\|_{X(I)}^{1-\theta}\nonumber\\
&\quad +\epsilon^{(\frac{s_c}{d}-\alpha_0)(1-\theta)}\big\|\psi\big\|_{X(I)}^{\theta}+\epsilon^{\frac{s_c}{2d}(1-\theta)}\big\|\psi\big\|_{X(I)}\Big)^p.\label{15.59-ex}
\end{align}
Thus by \eqref{17.34}, \eqref{15.59-ex} and \eqref{10.02}, we have
\begin{align*}
\epsilon^{-\frac{s_c}{2d}}\big\||\nabla|^{\gamma}\psi\big\|_{L^{q_\gamma}_{tx}(I\times \R^d)}
\lesssim & \epsilon^{1-\frac{s_c}{2d}}+\big\|\psi\big\|_{X(I)}\cdot \Big(\epsilon^{\frac{s_c}{d}(1-\theta)-\alpha_0}+\epsilon^{-\alpha_0\theta+\frac{s_c}{2d}(1-\theta)}\big\|\psi\big\|_{X(I)}^{1-\theta}\\
&\quad +\epsilon^{(\frac{s_c}{d}-\alpha_0)(1-\theta)}\big\|\psi\big\|_{X(I)}^{\theta}+\epsilon^{\frac{s_c}{2d}(1-\theta)}\big\|\psi\big\|_{X(I)}\Big)^p.
\end{align*}
Choosing $\alpha_0$ and $\varepsilon_0$ small enough, and using the Cauchy-Schwartz inequality, we obtain that
\begin{align}
\epsilon^{-\frac{s_c}{2d}}\big\||\nabla|^{\gamma}\psi\big\|_{L^{q_\gamma}_{tx}(I\times \R^d)}
\lesssim & \epsilon^{1-\frac{s_c}{2d}}+\epsilon^{\frac{s_cp}{4d}}\big\|\psi\big\|_{X(I)}+\big\|\psi\big\|_{X(I)}^{1+p}.\label{12.21}
\end{align}

For the other two norms in $X(I)$, one has
\begin{align*}
\big\||\nabla|^{s_c}\psi\big\|_{L^\infty_{t}L^2_x(I\times \R^d)}&+\big\|\psi\big\|_{L^{\frac{(d+2)p}{2}}_{tx}(I\times \R^d)}\\
\lesssim &\big\|\psi_0\big\|_{\dot H^{s_c}(\R^d)}+\big\||\nabla|^{s_c}(|u|^pu)\big\|_{L^{\frac{2d+4}{d+4}}_{tx}(I\times \R^d)}\\
\lesssim &\big\|\psi_0\big\|_{\dot H^{s_c}(\R^d)}+\big\||\nabla|^{s_c}u\big\|_{L^{\frac{2d+4}{d}}_{tx}(I\times \R^d)}\big\|u\big\|_{L^{\frac{(d+2)p}{2}}_{tx}(I\times \R^d)}^p.
\end{align*}
Notice that by \eqref{17.20}, one has
$$
\big\||\nabla|^{s_c}u\big\|_{L^{\frac{2d+4}{d}}_{tx}(I\times \R^d)}
\lesssim \big\||\nabla|^{s_c}h_L\big\|_{L^{\frac{2d+4}{d}}_{tx}(I\times \R^d)}+\big\||\nabla|^{s_c}\psi\big\|_{L^{\frac{2d+4}{d}}_{tx}(I\times \R^d)}
\lesssim\epsilon^{-\alpha_0}+\big\|\psi\big\|_{X(I)}
$$
and
$$
\big\|u\big\|_{L^{\frac{(d+2)p}{2}}_{tx}(I\times \R^d)}
\lesssim\epsilon^{\frac{s_c}{d}-\alpha_0}+\epsilon^{\frac{s_c}{2d}}\big\|\psi\big\|_{X(I)}.
$$
Then by using similar argument as above, we have
\begin{align}
\big\||\nabla|^{s_c}\psi\big\|_{L^\infty_{t}L^2_x(I\times \R^d)}&+\big\|\psi\big\|_{L^{\frac{(d+2)p}{2}}_{tx}(I\times \R^d)}\notag\\
\lesssim & \epsilon+\epsilon^{\frac{s_cp}{2d}}+\epsilon^{\frac{s_cp}{4d}}\|\psi\|_{X(I)}+\|\psi\|_{X(I)}^{1+p}.\label{12.31}
\end{align}

 It follows from \eqref{12.21} and \eqref{12.31} that
\begin{align*}
\|\psi\|_{X(I)}\lesssim \epsilon^{1-\frac{s_c}{2d}}+\epsilon+\epsilon^{\frac{s_cp}{2d}}+\epsilon^{\frac{s_cp}{4d}}\|\psi\|_{X(I)}+\|\psi\|_{X(I)}^{1+p}.
\end{align*}
Hence, for some $a_0>0$,
\begin{align*}
\|\psi\|_{X(I)}\lesssim \epsilon^{a_0}+\|\psi\|_{X(I)}^{1+p}.
\end{align*}
Therefore, by the bootstrap argument, we obtain that
\begin{align*}
\big\|\psi\big\|_{X(I)}
\lesssim \epsilon^{a_0}.
\end{align*}
One can see that the estimate above is uniformly in interval $I\subset \R$, which gives the global existence of the solution. Further, the scattering statement can be proved in the same way as in the proof of Theorem \ref{thm:main2}.
This finishes the proof of Theorem \ref{thm:main3}.

\section*{Acknowledgements}

Part of this work was done while M. Beceanu, A. Soffer and Y. Wu were visiting CCNU (C.C.Normal University) Wuhan, China. The authors thank the institutions for their hospitality and the support. M.B. is partially supported by NSF grant DMS 1700293.  Q.D.is partially supported by NSFC 11661061 and 11771165, and the Fundamental Research Funds for the Central Universities (CCNU18QN030, CCNU18CXTD04).  A.S is partially supported by NSF grant DMS 01600749 and NSFC 11671163. Y.W. is partially supported by NSFC 11771325 and 11571118. A.S. and Y.W. would like to thank  Ch. Miao for useful discussions.

\end{document}